\documentclass[a4paper,11pt,reqno]{amsart}
\usepackage{geometry} 
\usepackage{mathtools}
\usepackage{graphicx}
\usepackage{paralist}
\usepackage[utf8]{inputenc}
\usepackage[T1]{fontenc}
\usepackage{csquotes}
\usepackage[all]{xy}
\usepackage{chngcntr}
\usepackage{amsthm}
\usepackage{amsmath}
\usepackage{amsfonts}
\usepackage{amssymb}
\usepackage{mathrsfs}
\usepackage{MnSymbol}
\usepackage{tikz-cd}
\usepackage{tikz, tikz-cd}
\usepackage{mathdots}
\usepackage{stmaryrd}
\usepackage{hyperref}
\usepackage{cleveref}
\usepackage{fancyhdr}
\usepackage{xcolor}

\newtheorem{theorem}{Theorem}
\newtheorem{thm}{Theorem}[section]
\newtheorem{lem}[thm]{Lemma}
\newtheorem{prop}[thm]{Proposition}
\newtheorem{cor}[thm]{Corollary}
\newtheorem{rem}[thm]{Remark}
\theoremstyle{remark}

\theoremstyle{claim}

\newtheorem*{thm*}{Theorem}

\theoremstyle{definition}
\newtheorem{definition}[thm]{Definition}

\theoremstyle{example}

\theoremstyle{convention}
\newtheorem{convention}[thm]{Convention}

\counterwithin{equation}{section}

\theoremstyle{convention}

\newcommand{\bbN}{\mathbb{N}}

\newcommand{\bbR}{\mathbb{R}}

\newcommand{\bbH}{\mathbb{H}}

\DeclareMathOperator{\Ai}{Ai}
\DeclareMathOperator{\Bi}{Bi}

\DeclareMathOperator{\jac}{jac}
\newcommand{\jaco}{\jac_{\circ}}

\newcommand{\Vo}{\mathcal{V}_{\circ}}
\newcommand{\Ao}{A_{\circ}}
\newcommand{\Lo}{\mathcal{L}_{\circ}}
\newcommand{\vo}{v^{\circ}}
\newcommand{\alphao}{\alpha^{\circ}}

\newcommand{\tildeOmega}{\widetilde{\Omega}}

\newcommand{\ho}{h^\circ}

\newcommand{\tildeDelta}{\widetilde{\Delta}}
\newcommand{\tildeDeltao}{\widetilde{\Delta}_{\circ}}
\newcommand{\lambdaODE}{\tilde{\lambda}^{\rm ODE}}
\newcommand{\lambdao}{\tilde{\lambda}^{\circ}}
\newcommand{\lambdacirc}{\lambda^{\circ}}

\newcommand{\uGuess}{\tilde{u}^{\rm Guess}}
\newcommand{\uo}{\tilde{u}^{\circ}}
\newcommand{\hGuess}{\tilde{h}^{\rm Guess}}

\DeclareMathOperator{\Div}{div}
\newcommand{\Go}{G_{\circ}}

\calclayout

\begin{document}
\title[Constant potentials do not minimize in negatively curved manifolds]
{Constant potentials do not minimize the fundamental gap on convex domains in negatively curved Hadamard manifolds}
\author{Frieder J\"ackel}
\address{Département de Mathématiques \\ Universit\'e Libre de Bruxelles \\ Campus de la Plaine, CP 210, Boulevard du Triomphe, B-1050 Bruxelles, Belgique}
\email{frieder.jackel@ulb.be}
\thanks
{
The author was supported by the grant FNRS EoS O.0033.22F of the Fonds de la Recherche Scientifique (F.R.S.-FNRS)}
\date{February 2, 2026}

\begin{abstract}
We show that for every negatively curved Hadamard manifold $X$ and every $D > 0$ there exists a convex domain $\Omega \subseteq X$ with diameter $D$ and a convex potential $V$ on $\Omega$ such that the fundamental gap of the operator $-\Delta+V$ is strictly smaller than the fundamental gap of $-\Delta$. This shows that the second part of the fundamental gap conjecture is wrong in every negatively curved manifold. This is significantly harder than in the previously known case of hyperbolic space because, due to the lack of symmetry, one has to study a true PDE, and not just an ODE.
\end{abstract}

\maketitle

\tableofcontents

\newpage

\section{Introduction}\label{sec: Introduction}

\subsection{Statement of the main result}\label{subsec: Main result}

Understanding the relation between the spectrum of the Laplacian and the geometry of the underlying domain is a fundamental problem in Geometric Analysis with a long and fruitful history. In this article we study a spectral invariant that has attracted a lot of attention in recent years, the so-called fundamental gap.

Namely, let $X$ be a complete Riemannian manifold, $\Omega \subseteq X$ a connected, open, and bounded domain, and $V$ a potential on $\Omega$. Then, the \textit{fundamental gap} $\Gamma(\Omega;V)$ is the difference between the first two Dirichlet eigenvalues of the operator $-\Delta+V$, i.e.,
\[
	\Gamma(\Omega;V) \coloneq \lambda_2-\lambda_1 > 0,
\]
where $\lambda_1 < \lambda_2 \leq \dots$ are the eigenvalues of $-\Delta+V$ on $\Omega$ with Dirichlet boundary condition. If the potential is constant, we simply write $\Gamma(\Omega)$.

In their celebrated work \cite{AC11} on the solution of the Fundamental Gap Conjecture, Andrews--Clutterbuck have shown that for every bounded \emph{convex} domain $\Omega \subseteq \bbR^n$ and every \emph{convex} potential $V$ on $\Omega$ we have
\[
	\frac{3\pi^2}{{\rm diam}(\Omega)^2} \leq \Gamma(\Omega) \leq \Gamma(\Omega;V).
\]
We refer the reader to the introduction of \cite{AC11} for more information about the history and earlier work on this subject.

The main result of this article shows that the second part of the fundamental gap conjecture is wrong in every negatively curved Hadamard manifold for general convex domains and convex potentials.

\begin{theorem}\label{Main Theorem}
Let $X$ be a Hadamard manifold with negative sectional curvature. Then, for every $D > 0$, there exists a \emph{convex} domain $\Omega \subseteq X$ and a \emph{convex} potential $V$ on $\Omega$ such that ${\rm diam}(\Omega)=D$ and $\Gamma(\Omega;V) < \Gamma(\Omega)$.
\end{theorem}

This was recently shown by the author and Clutterbuck--Nguyen \cite{CJN25} in the case of hyperbolic space. However, the general case is much harder than the hyperbolic case. The reason for this is the following:~Due to the large symmetry group of $\bbH^n$, one can choose convex domains $\Omega \subseteq \bbH^n$ that allow a separation of variables ansatz. Consequently, the Dirichlet eigenvalue problem for these domains reduces to an ODE eigenvalue problem. 
In contrast, due to the lack of symmetries, in a general negatively curved Hadamard manifold it is impossible to directly reduce the problem to an ODE, and one is forced to study a true PDE. Nonetheless, we build on the ideas developed in \cite{CJN25}.

After the solution of the Fundamental Gap Conjecture in $\bbR^n$ by Andrews--Clutterbuck, a tremendous amount of attention has been given to the first part of the Fundamental Gap conjectures, i.e., whether $\Gamma(\Omega)$ or $\Gamma(\Omega;V)$ can be bounded from below in terms of the diameter, in manifolds different from $\bbR^n$. Namely, Dai, He, Seto, Wang, and Wei (in various subsets) \cite{SWW19}, \cite{DSW21}, \cite{HWZ20} have shown that the same bound $\Gamma(\Omega;V) \geq \frac{3\pi^2}{{\rm diam}(\Omega)^2}$ holds for convex subsets $\Omega \subseteq \mathbb{S}^n$ (though a stronger version than convexity is required for non-constant potentials - \cite[Equation (1.4)]{CWY25}). On the other hand, Bourni--Clutterbuck--Nguyen--Stancu--Wei--Wheeler \cite{BCN+22} showed that the first part of the Fundamental Gap Conjecture fails in a very strong sense in $\bbH^n$. Namely, they showed that for all $n \geq 2$ and $D,\varepsilon > 0$ there exists a convex domain $\Omega \subseteq \bbH^n$ of diameter $D$ such that $\Gamma(\Omega) < \frac{\varepsilon}{D^2}$. This was later extended to all negatively curved Hadamard manifolds by Khan--Nguyen \cite{KN24}. However, we still have $\Gamma(\Omega) \geq c(n,D) > 0$ for \emph{horo-convex} domains $\Omega \subseteq \bbH^n$ as was shown by Khan--Saha--Tuerkoen \cite{KST25} and Khan--Tuerkoen \cite{KT24} (an alternative proof for horo-convex domains of small diameter was recently given by Wei--Xiao \cite{WX25}).

In contrast, almost no attention has been given to the second part of the Fundamental Gap Conjecture on manifolds different from $\bbR^n$, i.e., whether $\Gamma(\Omega;V)$ is, among all convex potentials, minimized by the constant ones. In fact, apart from some work in the graph case \cite{AAHK24}, the recent \cite{CJN25} was the first one to adress this question after \cite{AC11}. The present article is, to the knowledge of the author, thus only the second one to do so after \cite{AC11}, and it deals with a much broader class of manifolds than \cite{CJN25}.

When proving the failure of the first part of the Fundamental Gap Conjecture in hyperbolic space (see \cite{BCN+22}) or in general negatively curved Hadamard manifolds (see \cite{KN24}), the key point is to find domains $\Omega$ such that the first eigenfunction of $-\Delta$ is very small in the so-called neck region of these domains. This is, however, much weaker than what is required for the failure of the second part (see (\ref{eq: core of strategy}) below). 

\subsection{Strategy}\label{subsec: strategy}
At its core, we use the same variational strategy as in \cite{CJN25}. Namely, we will construct a convex domain $\Omega \subseteq X$ and a convex potential $P$ on $\Omega$ such that
\begin{equation}\label{eq: core of strategy}
	\int_\Omega P (u_2^2-u_1^2) \, d{\rm vol} < 0,
\end{equation}
where $u_k$ ($k=1,2$) is the $k$-th eigenfunction of the Laplace operator $-\Delta$ on $\Omega$ with Dirichlet boundary condition. Namely, if (\ref{eq: core of strategy}) is satisfied, then, for $r > 0$ small enough, the fundamental gap of the potential $V_r=r P$ is strictly smaller than that of the constant potential. Indeed, by the Hellmann--Feynman identity, the derivative $\left. \frac{d}{dr}\Gamma(\Omega;V_r)\right|_{r=0}$ is exactly the integral in (\ref{eq: core of strategy}) (one can apply the Hellmann--Feynman identity thanks to \cite[Theorem 3.9 on p.~392]{Ka95}). We mention that, for our chosen domains, the first and second eigenvalue will turn out to be simple (see (\ref{eq: eigenvalue gap PDE})).



Even though in general negatively curved Hadamard manifolds the eigenvalue problem for $-\Delta$ can \emph{not} directly be reduced to an ODE, the main insight of the present article is that, for appropriately chosen domains, the eigenfunctions of $-\Delta$ can still be well-approximated by a separation of variables ansatz.

The basic idea to achieve this is roughly as follows:~We choose a sequence of domains $\Omega_\varepsilon \subseteq X$ that concentrate more and more around a fixed geodesic $\gamma_o$ as $\varepsilon \to 0$. This allows us to approximate $\Delta$ in $\Omega_\varepsilon$ by a elliptic operator $\Delta_\circ$ whose coefficients, in a well-chosen local coordinate system, only depend on the parameter $t$ parametrizing the geodesic $\gamma_o$. We then define a \emph{model ODE} for $-\Delta_\circ$, and show that
\begin{enumerate}[(1)]
	\item the model ODE is, after appropriate rescaling, a small perturbation of the Airy equation, so that its eigenfunctions are also well-approximated by the eigenfunctions of the Airy equation, and
	\item that the eigenfunctions of $\Delta_\circ$ can be obtained, up to small errors, by a separation of variables ansatz from the eigenfunctions of the model ODE.
\end{enumerate}
After correctly exploiting the assumption that $X$ is negatively curved (see \Cref{lem: Monotonicity of mu_1}), (1) can be achieved by adapting the arguments from \cite{CJN25}. On the other hand, (2) requires new ideas, and we refer the reader to the comments after \Cref{prop: Separation of variables} for an outline thereof. Using properties of the eigenfunctions of the Airy equation, we will be able to verify (\ref{eq: core of strategy}) for any potential $P$ that is strictly increasing along $\gamma_o$.

\subsection{Structure of the article}
The article is structured as follows. In \Cref{sec: set up} we properly introduce the necessary objects. Namely, in \Cref{subsec: domain and potential} we define the convex domains and convex potentials we will consider, while \Cref{subsec: operators and function spaces} introduces the operators $\Delta,\Delta_\circ$, their appropriate rescalings $\tildeDelta,\tildeDeltao$, and the associated function spaces. In \Cref{sec: Preliminaries} we collect the necessary background material about perturbation theory (\Cref{subsec: perturbation estimates}) and the Airy equation (\Cref{subsec: Airy}). The heart of the article is contained in \Cref{sec: Laplace_0}. Namely, after introducing the necessary objects in \Cref{subsec: eigenfcts of Laplace_0 - set up}, we introduce and analyze the model ODE in \Cref{subsec: Model ODE}, and then, in \Cref{subsec: Separation of variables}, show that the eigenfunctions of $\tildeDeltao$ are up to small errors obtained by a separation of variables Ansatz from the eigenfunctions of the model ODE. \Cref{subsec: from Laplace_0 to Laplace} explains why the eigenfunctions of $\tildeDelta$ are well-approximated by those of $\tildeDeltao$. Finally, the proof of \Cref{Main Theorem} will be presented in \Cref{subsec: Proof of Main Thm}.

\section{Set Up}\label{sec: set up}

\subsection{Domain and potential}\label{subsec: domain and potential}

Let $X$ be a Hadamard manifold with negative sectional curvature. Fix a basepoint $o \in X$ and a codimension one linear subspace $H \subseteq T_oX$, and consider the submanifold $\Sigma \coloneq \exp_o(H) \subseteq X$. Then
\begin{equation}\label{eq: Sigma 2nd fundamental form}
\text{the second fundamental form of } \Sigma \text{ vanishes at } o.
\end{equation}
Finally, we also fix a unit normal vector field $\nu$ along $\Sigma$. 


\begin{definition}[Domain]\label{def: Domain}
For all $t_0 >0$, $ \tau_0 \geq 0 $, and $\varepsilon > 0$ we define
\[
	\Omega \coloneq \Omega(\varepsilon; -\tau_0,t_0) \coloneq \Big\{\exp_{p}(t \nu_p) \in X \, \big| \, p \in B^{\Sigma}_\varepsilon(o) \text{ and } t \in (-\tau_0,t_0) \Big\},
\]
where $B^{\Sigma}_\varepsilon(o)$ denotes the ball of radius $\varepsilon$ centered at $o$ in $\Sigma$. 
\end{definition}

The next lemma shows that, for appropriate choices, these domains are convex.

\begin{lem}\label{lem: convex domain}
For all $t_0 > 0$ and $\tau_0 > 0$ there exists $\varepsilon_0=\varepsilon_0(t_0,\tau_0) > 0$ such that the domain $\Omega(\varepsilon; -\tau_0,t_0)$ is convex for all $\varepsilon \in (0,\varepsilon_0]$.
\end{lem}

If $\Sigma \subseteq X$ were totally geodesic in every point, then $\Omega(\varepsilon;0,t_0)$ would be convex for all $\varepsilon >0$ and $t_0 > 0$ since $\sec(X)<0$.

The following technical remark will be important in the proof of \Cref{Main Theorem}.

\begin{rem}[Locally uniform constants]\label{rem: uniform constants}\normalfont
The constant $\varepsilon_0$ in \Cref{lem: convex domain} can be chosen to be locally uniform in $t_0 > 0$ and $\tau_0 > 0$, i.e., for all $0 < t_{-} \leq t_{+} < \infty $ and $0 < \tau_{-} \leq \tau_{+} < \infty $ there exists $\varepsilon_0=\varepsilon_0(t_{-},t_{+},\tau_{-},\tau_{+}) > 0$ such that the conclusion of \Cref{lem: convex domain} holds for all $\varepsilon \in (0,\varepsilon_0]$, and all $t_0 \in [t_{-},t_{+}]$ and $\tau_0 \in [\tau_{-},\tau_{+}]$. In fact, this will be the case for all constants throughout this article. Indeed, this will be clear from the proofs.
\end{rem}

\begin{proof}
The boundary of $\Omega(\varepsilon; -\tau_0,t_0)$ decomposes into the three parts
\[
	\partial \Omega(\varepsilon; -\tau_0,t_0) = \Sigma_{t_0}^{\varepsilon} \cup Z_\varepsilon \cup \Sigma_{-\tau_0}^{\varepsilon},
\]
where $\Sigma_{t_0}^{\varepsilon}$ resp.~$\Sigma_{-\tau_0}^{\varepsilon}$ are the top resp.~bottom of the boundary, i.e.,
\begin{align*}
	\Sigma_{t_0}^{\varepsilon} \coloneq  
	\left\{\exp_{p}(t_0 \nu_p) \in X \, | \, p \in B^{\Sigma}_\varepsilon(o)\right\} \quad \text{resp.} \quad
	\Sigma_{-\tau_0}^{\varepsilon} \coloneq 
	\left\{\exp_{p}(-\tau_0 \nu_p) \in X \, | \, p \in B^{\Sigma}_\varepsilon(o)\right\},
\end{align*}
and $Z_\varepsilon$ is the cylindrical part of the boundary given by 
\[
	Z_\varepsilon \coloneq \left\{\exp_{p}(t \nu_p) \in X \, | \, p \in \partial B^{\Sigma}_\varepsilon(o) \text{ and } t \in (-\tau_0,t_0) \right\}.
\]
Since $Z_\varepsilon$ intersects $\Sigma_{t_0}^{\varepsilon}$ and $\Sigma_{-\tau_0}^{\varepsilon}$ orthogonally, it suffices to show that for each of these parts of the boundary, their shape operator with respect to the outward pointing unit normal is positive.

It is clear that for $\varepsilon > 0$ small enough (depending on $t_0$ and $\tau_0$), the shape operator of $Z_\varepsilon$ with respect to the outward pointing unit normal is positive. 

It remains to check that the same is true for $\Sigma_{t_0}^\varepsilon$ (the argument for  $\Sigma_{-\tau_0}^\varepsilon$ is analogous and will not be carried out). We denote, for all $t \in \bbR$, by 
\[
	\Sigma_t \coloneq \big\{\exp_p(t \nu_p) \in X \, | \, p \in B_\varepsilon^\Sigma(o)\big\}
\]
the codimension one submanifold with signed distance $t$ to $\Sigma$. Moreover, for all $t \neq 0$, denote by $\mathcal{H}_t$ the shape operator of $\Sigma_t$ with respect to the outward pointing unit normal $\nabla d(\cdot,\Sigma)$. Consider the geodesic $\gamma_o(t) \coloneq \exp_o(t \nu_o)$, where $o \in \Sigma$ is the fixed basepoint. Then $\lim_{t \to 0}\mathcal{H}_t|_{\gamma_o(t)}=0$ thanks to (\ref{eq: Sigma 2nd fundamental form}). It follows from Riccati comparison that $\mathcal{H}_t|_{\gamma_o(t)}$ is positive definite for all $t > 0$. In particular, $\mathcal{H}_{t_0}|_{\gamma_o(t_0)}$ is positive definite. Therefore, by continuity, $\mathcal{H}_{t_0}$ is positive definite on $\Sigma_{t_0}^\varepsilon$ for all $\varepsilon > 0$ small enough (depending on $t_0$). This completes the proof.
\end{proof}

To define the convex potential $P$ on $\Omega(\varepsilon;-\tau_0,t_0)$, we fix another $t_P \in (t_0/4,t_0/2)$.

\begin{definition}[Potential]\label{def: Potential}
For all $t_0 > 0$, $\tau_0 \geq 0$, $\varepsilon > 0$ and $t_P \in (t_0/4,t_0/2)$ define
\[
	P \, \colon \Omega(\varepsilon;-\tau_0,t_0) \to \bbR, \, \exp_p(t \nu_p) \mapsto
	\begin{cases}
		0 & \text{ if } -\tau_0 \leq t \leq t_P \\
		(t-t_P)^3 & \text{ if } t_P \leq t \leq t_0
	\end{cases}.
\]
\end{definition}

The precise formula for $P$ is \emph{not} important. The crucial property is that $P( \exp_p(t\nu_p))$ only depends on $t$. To show that $P$ is convex on $\Omega$, we will only need that $t \mapsto P(t)$ is $C^2$, $P=0$ for $t \leq t_P$, and  $P^\prime \geq 0$ and $P^{\prime \prime} \geq 0$. After that, we will only need that $t \mapsto P(t)$ is $C^2$ and $P^\prime(t_0) > 0$ (see \Cref{cor: asymptotic integral}). So there is a lot of freedom in the choice of $P$.

\begin{lem}\label{lem: convex potential}
For all $t_0 > 0$  there exists $\varepsilon_0=\varepsilon_0(t_0) > 0$ such that the potential $P$ is convex on $\Omega(\varepsilon; -\tau_0,t_0)$ for all $\varepsilon \in (0,\varepsilon_0]$.
\end{lem}

\begin{proof}
For ease of notation we abbreviate $\Omega \coloneq \Omega(\varepsilon; -\tau_0,t_0)$.
We use the same notation $\Sigma_t$ and $\mathcal{H}_t$ as in the proof of \Cref{lem: convex domain}. Moreover, we denote by $\partial_t$ the unit normal vector field on $\Sigma_t$, i.e.,
\[
	\partial_t|_{\exp_p(t\nu_p)} \coloneq \left.\frac{d}{d \tau}\exp_p(\tau \nu_p)\right|_{\tau = t}.
\]
Using that $P$ is $C^2$ and only depends on $t$, an easy computation shows that, for all $v,w \in T \Sigma_t$, we have
\[
	(\nabla^2 P) (\partial_t,\partial_t)=\partial^2_{tt}(P),
	\quad 
	(\nabla^2 P)(\partial_t,v)=0,
	\quad \text{and} \quad
	(\nabla^2 P)(v,w)=\langle \mathcal{H}_t(v),w \rangle \partial_t(P).
\]
By construction we have $P^{\prime \prime}(t) \geq 0$ and $P^\prime(t) \geq 0$ for all $t \in [-\tau_0,t_0]$, and $P(t)=0$ for all $t \in [-\tau_0,t_P]$. Moreover, it follows from the proof of \Cref{lem: convex domain} that $\mathcal{H}_t$ is positive definite on $\Sigma_t \cap \Omega$ for all $t \in [t_0/4,t_0]$ and all $\varepsilon > 0$ small enough (depending on $t_0$). Therefore, $\nabla^2 P \geq 0$ in $\Omega$ for all $\varepsilon > 0$ small enough. This completes the proof.
\end{proof}

Later, we will first fix $t_0 > 0$, then choose an appropriate $\tau_0 > 0$ (depending on $t_0$), and only consider $\varepsilon > 0$ small enough such that \Cref{lem: convex domain} and \Cref{lem: convex potential} can be applied.

\subsection{Differential operators and their function spaces}\label{subsec: operators and function spaces}
We continue to use the same notation as in \Cref{subsec: domain and potential}, i.e., $H \subseteq T_oX$ is a fixed linear subspace of codimension one, $\Sigma \coloneq \exp_o(H)$, and $\nu$ is a unit normal vector field along $\Sigma$.

Denote by $\psi_{\Sigma} \colon H \to \Sigma$ the restriction of $\exp_o$ to $H$, and fix an identification $\bbR^{n-1} \cong H$. We define
\begin{equation*}\label{def coordinates}
	\Psi \colon  \bbR^{n-1} \times \bbR, \, (s,t) \longmapsto \exp_{\psi_{\Sigma}(s)}\big(t \nu_{\psi_\Sigma(s)} \big).
\end{equation*}
Then, $\Psi$ is a local diffeomorphism around $(0,0)$. More precisely, for all $t_0, \tau_0 > 0$, there exists $\epsilon_0 > 0$ such that the restriction of $\Psi$ to $ B_{\epsilon_0}(0) \times (-\tau_0,t_0)$ is a diffeomorphism onto its image. So, for all $\varepsilon > 0 $ small enough, $\Psi$ is a diffeomorphism from $ B_\varepsilon(0) \times (-\tau_0,t_0)$ to $\Omega(\varepsilon; -\tau_0,t_0)$. It will be convenient to use $\Psi$ as a coordinate chart on a slightly larger domain, e.g., on $ B_{\epsilon_0}(0) \times (-\tau_0-1,t_0+1) \cong \Omega(\epsilon_0;-\tau_0-1,t_0+1)$ for some $\epsilon_0 > 2\varepsilon$.

Note that, in the coordinate chart $\Psi$, the metric $g$ of $X$ is of the form
\begin{equation}\label{eq: def metric g}
	g=dt^2 + g_{ij}(s,t) ds^i ds^j + g_{ti}(s,t) dt ds^i,
\end{equation}
with $g_{ti}(0,t)=0$ for all $t \in (-\tau_0,t_0)$ as a consequence of (\ref{eq: Sigma 2nd fundamental form}). 

We define $g_{\circ}$ as the Riemannian metric on $X$ obtained from $g$ by evaluating $s=0$ everywhere, i.e.,
\begin{equation}\label{eq: def metric g_0}
	g_{\circ} \coloneq dt^2 + g_{ij}(0,t) ds^i ds^j.
\end{equation}

We denote by $\Delta$ the Laplacian of $g$, and by $\Delta_\circ$ the Laplacian of $g_\circ$ on $X$. In local coordinates they are given by 
\begin{equation}\label{eq: def Laplace}
	\Delta = \frac{1}{\sqrt{|g|}}\bigg(\partial_t\Big(\sqrt{|g|} \partial_t \Big)
	+\partial_{s_i}\Big(\sqrt{|g|} g^{ij}\partial_{s_j}\Big)
	+\partial_{t}\Big(\sqrt{|g|} g^{ti}\partial_{s_i}\Big)
	+\partial_{s_i}\Big(\sqrt{|g|} g^{ti}\partial_{t}\Big) \bigg),
\end{equation}
and similarly, as $(g_{\circ})_{ij}$ only depends on $t$, 
\begin{equation}\label{eq: def Laplace_0}
	\Delta_{\circ} = \frac{1}{\sqrt{|g_{\circ}|}}\partial_t\Big(\sqrt{|g_{\circ}|} \partial_t \Big)
	+\partial_{s_i}\Big(g_{\circ}^{ij}\partial_{s_j}\Big).
\end{equation}
For ease of notation, we will abbreviate
\begin{equation}\label{eq: def jac and jac0}
	\jac(s,t) \coloneq \sqrt{|g|}(s,t)
	\quad \text{and} \quad 	
	\jaco(t) \coloneq \sqrt{|g_{\circ}|}(t).
\end{equation}

Most of the time, we will not directly work in the coordinates $(t,s)$ given by $\Psi$. Rather, we will use the change of variables 
\begin{equation}\label{eq: change of variables}
	x \coloneq \delta^{-1/3}(t_0-t)
	\quad \text{and} \quad
	y \coloneq \frac{s}{\varepsilon},
\end{equation}
where $\delta=\delta(\varepsilon) > 0$ is an appropriately chosen constant satisfying $\delta \to 0$ as $\varepsilon \to 0$ (the precise definition of $\delta$ is not important at the moment).

Similarly, instead of directly working with $\Delta$ and $\Delta_\circ$, we will investigate the rescaled operators
\begin{equation}\label{eq: def tilde Laplacians}
	\tildeDelta \coloneq \delta^{2/3} \Delta
	\quad \text{and} \quad
	\tildeDeltao \coloneq\delta^{2/3}\Delta_{\circ}.
\end{equation}
We consider $\tildeDelta$ and $\tildeDelta_0$ as elliptic partial differential operators on the domain
\begin{equation}\label{eq: def tilde(Omega)}
	\tildeOmega \coloneq \big( 0,\delta^{-1/3}(t_0+\tau_0) \big) \times B_1(0) \subseteq \bbR \times \bbR^{n-1}
\end{equation}
which corresponds to $\Omega$ under the coordinate chart $\Psi$ and the change of variables (\ref{eq: change of variables}).

Then $\tildeDelta$ resp.~$\tildeDeltao$ is an unbounded self-adjoint operator on the weighted $L^2$-space
\[
	\tilde{L}^2\big(\tildeOmega\big) \coloneq \left( L^2\big(\tildeOmega\big), \langle \cdot, \cdot \rangle_{\tilde{L}^2(\tildeOmega)} \right)
	\quad \text{resp.} \quad
	\tilde{L}_\circ^2\big(\tildeOmega\big) \coloneq \left( L^2\big(\tildeOmega\big), \langle \cdot, \cdot \rangle_{\tilde{L}_\circ^2(\tildeOmega)} \right),
\]
where $\langle \cdot, \cdot \rangle_{\tilde{L}^2(\tildeOmega)}$ resp.~$\langle \cdot, \cdot \rangle_{\tilde{L}_\circ^2(\tildeOmega)}$ is the weighted $L^2$-inner product given by
\begin{equation}\label{eq: def weighted L^2 norm xy}
	\langle w_1, w_2 \rangle_{\tilde{L}^2(\tildeOmega)} \coloneq 
	\int_{\tildeOmega}\frac{\jac(\varepsilon y,t_0-\delta^{1/3}x)}{\jaco(t_0)}w_1(x,y)w_2(x,y) \, dxdy
\end{equation}
resp.
\begin{equation}\label{eq: def weighted L_0^2 norm xy}
	\langle w_1, w_2 \rangle_{\tilde{L}_\circ^2(\tildeOmega)} \coloneq 
	\int_{\tildeOmega}\frac{\jaco(t_0-\delta^{1/3}x)}{\jaco(t_0)}w_1(x,y)w_2(x,y) \, dxdy.
\end{equation}
It follows from the following remark that on the domain $H^2(\tildeOmega) \cap H_0^1(\tildeOmega) \subseteq L^2(\tildeOmega)$ the operators $\tildeDelta$ and $\tildeDeltao$ are unbounded and self-adjoint.

\begin{rem}[$H^2$-estimates up to the boundary]\label{rem: H^2 up to boundary}\normalfont
Even though the boundary of $\tildeOmega$ is not smooth, we have $H^2$-estimates up to the boundary of $\tildeOmega$ for any elliptic operator in divergene form with Lipschitz coefficients.
\end{rem}

Indeed, using a standard odd-reflection construction along the bottom and top faces $\{0\} \times \bar{B}_1(0)$ and $\{\delta^{-1/3}(t_0+\tau_0)\} \times \bar{B}_1(0)$ of $\tildeOmega$, any solution $u \in H_0^1(\tildeOmega)$ of $\mathcal{L}u=f$ extends to a solution $\hat{u} \in H_0^1(\widehat{\Omega})$ of $\widehat{\mathcal{L}}\hat{u}=\hat{f}$ in the domain $\widehat{\Omega} = (-\delta^{-1/3}(t_0+\tau_0),2\delta^{-1/3}(t_0+\tau_0)) \times B_1(0)$. Thus, $H^2$-estimates up to the boundary of $\tildeOmega$ follow from interior $H^2$-estimates in $\widehat{\Omega}$.

\section{Preliminaries}\label{sec: Preliminaries}

\subsection{Perturbation estimates}\label{subsec: perturbation estimates}

In this subsection we gather elementary perturbation estimates that will be used in the remainder of this article.

The first result allows us to guess eigenvalues and eigenvectors. In its formulation, $\mathcal{V}$ is a real separable Hilbert space, $\mathcal{L}$ is an unbounded self-adjoint operator with compact resolvent on $\mathcal{V}$, and $(v_i)_{i \in \bbN} \subseteq \mathcal{V}$ is a complete orthonormal basis of eigenvectors of $\mathcal{L}$ with eigenvalues $\lambda_1 \leq \lambda_2 \leq \dots$.

\begin{lem}[Guessing Eigenobjects]\label{lem: Guessing eigenvectors}
Assume $v^{\rm Guess} \in {\rm Dom}(\mathcal{L})$ and $\lambda^{\rm Guess} \in \bbR$ and are such that
\[
	\left|\left|(\mathcal{L}-\lambda^{\rm Guess})v^{\rm Guess} \right| \right| \leq \varepsilon_{\rm Guess}
	\quad \text{and} \quad
	\left|\left|v^{\rm Guess} \right| \right| = 1
\]
for some $\varepsilon_{\rm Guess} \geq 0$. Then, there exists $j \in \bbN$ such that
\begin{equation}\label{eq: guessing eigenvalues}
	\left|\lambda_j-\lambda^{\rm Guess} \right| \leq \varepsilon_{\rm Guess}.
\end{equation}
Moreover, if the $j$-th eigenvalue $\lambda_j$ of $\mathcal{L}$ given by (\ref{eq: guessing eigenvalues}) is simple, then, after potentially changing $v_j$ up to sign, we have
\begin{equation}\label{eq: guessing eigenvectors}
	\left|\left|v_j-v^{\rm Guess} \right| \right| \leq \frac{C\varepsilon_{\rm Guess}}{\Gamma_j},
\end{equation}
where $\Gamma_j \coloneq \min\{\lambda_{j+1}-\lambda_j,\lambda_j-\lambda_{j-1}\} > 0$  and $C > 0$ is a universal constant.
\end{lem}

This is a straight-forward consequence of the spectral theorem.

\begin{proof}
If $R > 0$ is such that $|\lambda_j - \lambda^{\rm Guess}| \geq R$ for all $j \in \bbN$, then it follows from the spectral decomposition that $||(\mathcal{L}-\lambda^{\rm Guess})v|| \geq R||v||$ for all $v \in {\rm Dom}(\mathcal{L})$. From this, (\ref{eq: guessing eigenvalues}) follows immediately.

It remains to prove (\ref{eq: guessing eigenvectors}). If $\varepsilon_{\rm Guess} \geq \frac{1}{2}\Gamma_j$, then (\ref{eq: guessing eigenvectors}) trivially holds due to the triangle inequality. So we may from now on assume $\varepsilon_{\rm Guess} \leq \frac{1}{2}\Gamma_j$.

Decompose $v^{\rm Guess}=\beta v_j + w \in {\rm span}\{v_j\}+ {\rm span}\{v_j\}^\perp$. After potentially changing $v_j$ up to sign, we may without loss of generality assume $\beta \geq 0$. Then $||v^{\rm Guess}-v_j|| \leq \sqrt{2}||w||$, and thus it suffices to bound $||w||$. 

Observe that  $|\lambda_{j^\prime}-\lambda^{\rm Guess}| \geq \frac{1}{2}\Gamma_j$ for all $j^\prime \in \bbN \setminus \{j\}$ as a consequence of (\ref{eq: guessing eigenvalues}) and $\varepsilon_{\rm Guess} \leq \frac{1}{2}\Gamma_j$. Therefore, it follows from the spectral decomposition and $w \in {\rm span}\{v_j\}^\perp$ that
\(
	\frac{1}{2}\Gamma_j ||w|| \leq || (\mathcal{L}-\lambda^{\rm Guess})w ||.
\)
Finally, we can bound $|| (\mathcal{L}-\lambda^{\rm Guess})w || \leq 2\varepsilon_{\rm Guess}$ due to (\ref{eq: guessing eigenvalues}) and  
\begin{align*}
	(\mathcal{L}-\lambda^{\rm Guess})w =& (\mathcal{L}-\lambda_j)w +(\lambda_j-\lambda^{\rm Guess})w = (\mathcal{L}-\lambda_j)v^{\rm Guess}+(\lambda_j-\lambda^{\rm Guess})w \\
	=& (\mathcal{L}-\lambda^{\rm Guess})v^{\rm Guess}+(\lambda^{\rm Guess}-\lambda_j)(v^{\rm Guess}-w).
\end{align*}
This establishes (\ref{eq: guessing eigenvectors}), and thus completes the proof.
\end{proof}

The second perturbation estimate compares the eigenvalues and eigenvectors of two unbounded operators that are self-adjoint with respect to two different inner products. Namely,
let $\Vo=(V,\langle \cdot , \cdot \rangle_\circ)$ and $\mathcal{V}=(V, \langle \cdot, \cdot \rangle)$ be two real separable Hilbert space structures on a common underlying vector space $V$. We assume that $\langle \cdot , \cdot \rangle_\circ$ and $\langle \cdot , \cdot \rangle$ are almost equivalent in the following sense:~there exists $\varepsilon_{\rm comp} \in [0, 1/2]$ such that for all $w_1,w_2 \in V$ we have
\begin{equation}\label{eq: inner products almost equivalent}
	\big|\langle w_1 , w_2 \rangle_\circ - \langle w_1, w_2 \rangle \big|
	\leq \varepsilon_{\rm comp}||w_1||_\circ ||w_2||_\circ
\end{equation}
and the analogous bound with $||\cdot||$ instead of $||\cdot||_\circ$ in the upper bound.

Moreover, let $\Lo$ resp.~$\mathcal{L}$ be an non-negative unbounded self-adjoint operator with compact resolvent on $\Vo$ resp.~$\mathcal{V}$, and $(\vo_i)_{i \in \bbN} \subseteq \Vo$ resp.~$(v_i)_{i \in \bbN} \subseteq \mathcal{V}$ a complete orthonormal basis of eigenvectors of $\Lo$ resp.~$\mathcal{L}$ with eigenvalues $\lambdacirc_1 \leq \lambdacirc_2 \leq \dots$ resp.~$\lambda_1 \leq \lambda_2 \leq \dots$. We also assume $\vo_i \in {\rm Dom}(\mathcal{L})$ and $v_i \in  {\rm Dom}(\Lo)$ for all $i \in \bbN$. Finally, for all $k \in \bbN$, we abbreviate
\[
	S_k^{\circ} \coloneq {\rm span}\big\{\vo_1, \dots, \vo_k \big\}
	\quad \text{and} \quad
	S_k \coloneq {\rm span}\big\{v_1, \dots, v_k \big\}.
\] 
With this notation at hand, we can now formulate the second perturbation estimate. This is similar to \cite[Lemma 3.1]{CJN25}, but adapted to our purposes.

\begin{lem}\label{lem: perturbation}
Let $\Vo,\mathcal{V},\Lo,\mathcal{L}$ and $\varepsilon_{\rm comp}$ be as above. Then, for all $k \in \bbN$, we have
\begin{equation}\label{eq: eigenvalue perturbation}
	-3\varepsilon_{\rm comp}\lambda_k - \max_{v \in S_k}\frac{|\langle (\Lo-\mathcal{L})v,v \rangle_\circ |}{||v||_\circ^2}
	\leq
	\lambda_k-\lambdacirc_k 
	\leq 
	3\varepsilon_{\rm comp}\lambdacirc_k + \max_{v \in S_k^\circ}\frac{|\langle (\Lo-\mathcal{L})v,v \rangle |}{||v||^2}.
\end{equation}
Moreover, if the $k$-th eigenvalue $\lambdacirc_k$ of $\Lo$ is simple, then, after potentially changing $\vo_k$ up to sign, we have
\begin{equation}\label{eq: eigenvector perturbation}
	||v_k-\vo_k||^2_\circ \leq \frac{C}{\Gamma_k^\circ} \left(\varepsilon_{\rm comp}\lambdacirc_k+\max_{v \in S_k^\circ}\frac{|\langle (\Lo-\mathcal{L})v,v \rangle |}{||v||^2}+\max_{v \in S_k}\frac{|\langle (\Lo-\mathcal{L})v,v \rangle_\circ |}{||v||_\circ^2} \right),
\end{equation}
where $\Gamma_k^\circ \coloneq \min\big\{\lambdacirc_{k+1}-\lambdacirc_k,\lambdacirc_k-\lambdacirc_{k-1} \big\} > 0$ and $C$ is a universal constant. 
\end{lem}

Since the norms $||\cdot||_\circ$ and $|| \cdot ||$ are uniformly equivalent thanks to (\ref{eq: inner products almost equivalent}), we can replace $||v_k-\vo_k||^2_\circ$ by $||v_k-\vo_k||^2$ in (\ref{eq: eigenvector perturbation}).

\begin{proof}
We abbreviate $E \coloneq \mathcal{L}-\Lo$. Then, the min-max principle implies
\[
	\lambda_k \leq \max_{v \in S_k^\circ}\frac{\langle \mathcal{L}v, v \rangle}{||v||^2}
	\leq \max_{v \in S_k^\circ}\frac{\langle \Lo v, v \rangle}{||v||^2}+ \max_{v \in S_k^\circ}\frac{\langle Ev, v \rangle}{||v||^2}.
\]
Moreover, using (\ref{eq: inner products almost equivalent}) we can estimate
\[
	\langle \Lo v, v \rangle \leq \langle \Lo v, v \rangle_\circ + \varepsilon_{\rm comp}||\Lo v||_\circ ||v||_\circ.
\]
For all $v \in S_k^\circ$ we have $||\Lo v||_\circ \leq \lambdacirc_k ||v||_\circ$ since $\Lo$ is non-negative. Note that, due to (\ref{eq: inner products almost equivalent}), we have $||v||_\circ^2 \leq (1+\varepsilon_{\rm comp})||v||^2 \leq 2 ||v||^2$. Combining these estimates, we deduce
\begin{align*}
	\lambda_k \leq &  (1+\varepsilon_{\rm comp}) \max_{v \in S_k^\circ}\frac{\langle \Lo v, v \rangle_\circ}{||v||_\circ^2}+2\varepsilon_{\rm comp} \lambdacirc_k 
	+ \max_{v \in S_k^\circ}\frac{\langle Ev, v \rangle}{||v||^2} \\
	= & (1+\varepsilon_{\rm comp}) \lambdacirc_k+2\varepsilon_{\rm comp} \lambdacirc_k 
	+ \max_{v \in S_k^\circ}\frac{\langle Ev, v \rangle}{||v||^2},
\end{align*}
from which the upper bound in (\ref{eq: eigenvalue perturbation}) readily follows. Exchanging the roles of $\Lo$ and $\mathcal{L}$ also yields the lower bound, thus establishing (\ref{eq: eigenvalue perturbation}).

It remains to prove (\ref{eq: eigenvector perturbation}). Decompose $v_k=\beta\vo_k+ w^\circ \in {\rm span}\{\vo_k\}+{\rm span}\{\vo_k\}^\perp$. After potentially changing $\vo_k$ up to sign, we may assume $\beta \geq 0$. Then $||v_k-\vo_k||_\circ \leq \sqrt{2}||w^\circ||_\circ$, and thus it suffices to bound $||w^\circ||_\circ$.

Note $(\Lo-\lambdacirc_k)w^\circ = (\Lo-\lambdacirc_k)v_k$, $(\Lo-\lambdacirc_k)w^\circ \perp \vo_k$, and $\Lo v_k=\mathcal{L}v_k-Ev_k=\lambda_k v_k - Ev_k$. Hence
\[
	\left\langle (\Lo-\lambdacirc_k)w^\circ, w^\circ \right\rangle_\circ
	= \left\langle (\Lo-\lambdacirc_k)v_k, v_k \right\rangle_\circ
	= (\lambda_k-\lambdacirc_k)||v_k||_\circ^2 - \langle Ev_k, v_k \rangle_\circ
\]
However, $\Gamma_k^\circ ||w^\circ||_\circ^2 \leq \langle (\Lo-\lambdacirc_k)w^\circ, w^\circ \rangle_\circ$ since the restriction of $\Lo-\lambdacirc_k$ to ${\rm span}\{\vo_k\}^\perp$ has spectral gap $\Gamma_k$. Therefore, keeping in mind $||\cdot||_\circ^2 \leq 2||\cdot||^2$, (\ref{eq: eigenvector perturbation}) follows from (\ref{eq: eigenvalue perturbation}). 
\end{proof}

\subsection{Airy equation}\label{subsec: Airy}

We use this subsection to introduce notation and to collect some facts about the Airy equation. We refer to \cite[Section 2.2]{CJN25} for further details.

The \emph{Airy equation} is
\[
	y^{\prime \prime}(x)=xy(x).
\]
The Airy functions $\Ai$ and $\Bi$ are distinguished and linearly independent solutions of the Airy equation. Both are defined for all $x \in \bbR$, both are oscillating for $x < 0$, and for $x > 0$ $\Ai$ is exponentially decaying while $\Bi$ is exponentially growing (see for example \cite[Figure 3.1 on p.~69, and (3.5.17a),(3.5.17b) on p.~100]{BO99}). We denote by
\[
	\dots < -a_3 < -a_2 < -a_1 < 0
\]
the zeros of $\Ai$.

If $y:[0,\infty) \to \bbR$ is a non-zero solution of the Airy eigenvalue problem
\[
	y^{\prime \prime}(x)=(x-\alpha) y(x)
\]
with $y(0)=0$ and $y \in L^2(0,\infty)$, then $\Ai(-\alpha)=0$, i.e., $\alpha=a_k$ for some $k \in \bbN$. Therefore, for all $k \in \bbN$, the function
\begin{equation}\label{eq: def of Airy eigfct on half-line}
	v_k(x) \coloneq \frac{\Ai(x-a_k)}{||\Ai(\cdot - a_k)||_{L^2(0,\infty)}}
\end{equation}
is the $L^2$-normalized $k$-th eigenfunction of the Airy equation on $[0,\infty)$ with Dirichlet boundary conditions and eigenvalue $a_k$.

The following will serve as the (negative of the) asymptotic model for the integral in (\ref{eq: core of strategy}) (see \Cref{cor: asymptotic integral}).

\begin{lem}[Model Integral]\label{lem: model integral}
We have
\[
	\int_0^\infty x\left(v_2^2(x)-v_1^2(x) \right) \, dx = \frac{2}{3}(a_2-a_1) > 0,
\]
where $v_1,v_2$ are defined in (\ref{eq: def of Airy eigfct on half-line}), and $0 > -a_1 > -a_2 > \dots$ are the zeros of $\Ai$.
\end{lem}

This is proved in \cite[Lemma 2.2]{CJN25} using an easy scaling argument together with the Hellmann--Feynman identity.

\section{Eigenfunctions of \(\tildeDeltao\)}\label{sec: Laplace_0}

The goal of this section is to control the eigenfunctions of the operator $\tildeDeltao$ defined in \Cref{subsec: operators and function spaces}. The rough strategy to achieve this is as follows:
\begin{enumerate}[(1)]
	\item First, we introduce a model ODE for $\tildeDeltao$, and we show that the eigenfunctions of this model ODE are small perturbations of the eigenfunctions of the Airy equation.
	\item Then we show that, up to small errors, the eigenfunctions of $\tildeDeltao$ can be obtained from the eigenfunctions of the model ODE by a separation of variables ansatz.
\end{enumerate}
We refer to \Cref{prop: Model ODE and Airy} and \Cref{prop: Separation of variables} for the precise statements. 

To carry out this strategy, we first introduce the necessary notation and objects in \Cref{subsec: eigenfcts of Laplace_0 - set up}. Then, (1) is implemented in \Cref{subsec: Model ODE}. This can be achieved by a natural generalization of the arguments in \cite[Section 3.2]{CJN25} after correctly exploiting the fact that $X$ is negatively curved (see \Cref{lem: Monotonicity of mu_1}). On the other hand, the proof of (2) requires completely new arguments. This will be done in \Cref{subsec: Separation of variables}.

\subsection{Set up}\label{subsec: eigenfcts of Laplace_0 - set up}

We recall from (\ref{eq: def tilde(Omega)}) that $\tildeOmega$ denotes the region
\[
 	\tildeOmega = \big(0,\delta^{-1/3}(t_0+\tau_0)\big) \times B_1(0) \subseteq \bbR \times \bbR^{n-1},
\]
and that we consider $\tildeDeltao$ as an elliptic partial differential operator of second order on $\tildeOmega$. It follows from the formula (\ref{eq: def Laplace_0}) for $\Delta_\circ$ and the definition (\ref{eq: change of variables}) of the change of variables that $\tildeDeltao$ is given by
\begin{equation}\label{eq: def tilde(Laplace)_0}
	\tildeDeltao = \frac{\jaco(t_0)}{\jaco(t_0-\delta^{1/3}x)}\partial_x\left(\frac{\jaco(t_0-\delta^{1/3}x)}{\jaco(t_0)} \partial_x \right)
	+\delta^{2/3}\varepsilon^{-2}\partial_{y_i}\Big(g_{\circ}^{ij}(t_0-\delta^{1/3}x)\partial_{y_j}\Big).
\end{equation}
For ease of notation we abbreviate
\[
	\Go(t) \coloneq \big((g_\circ)_{ij}(t) \big)_{1 \leq i,j \leq n-1} \in {\rm GL}_{n-1}(\bbR),
\]
and simply write
\[
	\tildeDeltao = \frac{\jaco(t_0)}{\jaco}\partial_x\left(\frac{\jaco}{\jaco(t_0)} \partial_x \right)
	+\delta^{2/3}\varepsilon^{-2} \Div\left(\Go^{-1}\nabla \right),
\]
where $\Div$ and $\nabla$ are the standard divergence and gradient in $\bbR^{n-1}$. We will sometimes write $\Div_y$ and $\nabla_y$ to emphasize that they only act in the $y$-variable.

We also recall that $\jaco$ and $g_\circ^{ij}$ are defined in a region slightly larger than $\tildeOmega$, e.g., on $(-1,\delta^{-1/3}(t_0+\tau_0)+1) \times B_2(0)$.

Since, for all $t \in [-\tau_0,t_0]$, the first eigenvalue $\mu_1(t)$ of $-\Div(\Go^{-1}(t)\nabla)$ on $B_1(0) \subseteq \bbR^{n-1}$ with Dirichlet boundary condition is simple, $\mu_1(\cdot)$ is smooth, and there exists a smooth parametrization
\[
	f_1^{(t)} \in H_0^1\big(B_1(0) \big) \subseteq L^2\big(B_1(0)\big) 
	\quad \big(t \in [-\tau_0,t_0] \big)
\]
of the first $L^2$-normalized Dirichlet eigenfunction, i.e.,
\[
	-\Div\left(\Go^{-1}(t)\nabla f_1^{(t)} \right)=\mu_1(t)f_1^{(t)}
	\quad \text{and} \quad
	\left|\left|f_1^{(t)}\right|\right|_{L^2(B_1(0))}=1.
\]
By abstract perturbation theory, $t \mapsto f_1^{(t)}$ is a priori only smooth as a $L^2$-valued map. However, it easily follows from the Sobolev embeddings and elliptic regularity theory, that it is also smooth as a $C^m$-valued map for all $m \in \bbN$.

The following is the fundamental property of $\mu_1$ that will allow us to adapt the arguments from \cite{CJN25} to study the model ODE in \Cref{subsec: Model ODE}. Its proof will exploit that $X$ has negative sectional curvature in an essential way.

\begin{lem}[Monotonicity of $\mu_1$]\label{lem: Monotonicity of mu_1}
We have $\mu_1^\prime(t) < 0$ for every $t \in (0,t_0]$.
\end{lem}

Before we give the proof of \Cref{lem: Monotonicity of mu_1}, we define our choice of constants. First, we fix $t_0 > 0$. Then we only consider $\tau_0 > 0$ such that
\begin{equation}\label{eq: def tau}
	\mu_1(-\tau_0) \geq \mu_1(t_0/4).
\end{equation}
Such $\tau_0 > 0$ exists because of \Cref{lem: Monotonicity of mu_1}. Furthermore, we define $\delta=\delta(\varepsilon) > 0$ as
\begin{equation}\label{eq: def delta}
	\delta \coloneq \frac{\varepsilon^{2}}{-\mu_1^\prime(t_0)}.
\end{equation}
Note $\delta> 0$ thanks to \Cref{lem: Monotonicity of mu_1}.

We now give the proof of \Cref{lem: Monotonicity of mu_1}. 

\begin{proof}[Proof of \Cref{lem: Monotonicity of mu_1}]
The Hellmann-Feynman identity states that, for all $t \in [-\tau_0,t_0]$,
\begin{align*}
	\mu^\prime(t)=& \left\langle f_1^{(t)}, \left(\frac{d}{dt}\Div_y(\Go^{-1}(t)\nabla_y )\right)f_1^{(t)} \right\rangle_{L^2(B_1)} \\
	=&\int_{B_1(0)} \left\langle \left(\frac{d}{dt}\Go^{-1}(t)\right)\nabla_y f_1^{(t)}, \nabla_y f_1^{(t)}\right\rangle_{\bbR^{n-1}} \, dy.
\end{align*}
Taking the derivative of $\Go\Go^{-1}={\rm id}_{\bbR^{n-1}}$ implies $(\Go^{-1})^\prime=-\Go^{-1} \Go^\prime \Go^{-1}$. Thus, keeping in mind that $\Go^{-1}$ is symmetric, we obtain, for all $t \in [-\tau_0,t_0]$,
\begin{equation}\label{eq: monotonicity of mu 1}
	\mu_1^{\prime}(t)=-\int_{B_1(0)} \left\langle \Go^\prime(t)\left(\Go^{-1}(t)\nabla_y f_1^{(t)}\right), \Go^{-1}(t)\nabla_y f_1^{(t)}\right\rangle_{\bbR^{n-1}} \, dy.
\end{equation}
Recall the notation from \Cref{subsec: domain and potential} and the proof of \Cref{lem: convex domain}, i.e., $\Sigma_t \subseteq X$ is the codimension one submanifold with signed distance $t$ to $\Sigma$ and $\mathcal{H}_t$ is the shape operator of $\Sigma_t$. Consider the normal geodesic $\gamma_o(t) \coloneq \exp_o(t\nu_o)$ starting from the fixed basepoint $o \in \Sigma$. For all $i=1,\dots,n-1$, we denote by $J_i$ the Jacobi field along $\gamma_o$ with initial conditions $J_i(0)=\partial_{s_i}|_o$ and $J^\prime(0)=0$ (see (\ref{eq: Sigma 2nd fundamental form})). Then $(g_\circ)_{ij}=\langle J_i,J_j \rangle$, and thus
\[
	(g_\circ)_{ij}^\prime(t)=\langle J_i^\prime(t),J_j(t) \rangle +\langle J_i(t),J_j^\prime(t) \rangle
	= \langle \mathcal{H}_t(J_i),J_j \rangle + \langle J_i,\mathcal{H}_t(J_j) \rangle
	= 2\langle \mathcal{H}_t(J_i),J_j \rangle.
\]
Observe that $\mathcal{H}_t|_{\gamma_o(t)} \geq c(t){\rm id}_{T\Sigma_t}$ for a function $c(\cdot)$ with $c(t) > 0$ for $t > 0$. Indeed, this follows from Riccati comparison, $\sec(X) < 0$, and the fact that $\mathcal{H}_0|_o = 0$ by (\ref{eq: Sigma 2nd fundamental form}). Therefore, writing  $J_\xi \coloneq \xi^i J_i$ for $\xi=(\xi^1,\dots,\xi^{n-1}) \in \bbR^{n-1}$, we obtain, for all $t \in (0,t_0]$, 
\[
	\langle \Go^\prime(t)\xi,\xi \rangle_{\bbR^{n-1}}
	=2 \left\langle \mathcal{H}_t\left(J_\xi \right), J_\xi \right\rangle_X \geq 2c(t)\left|J_\xi \right|_X^2.
\]
Moreover, $|J_\xi|(t) \geq |J_\xi|(0)=|\xi|$ follows from Jacobi field comparison and $\sec(X) \leq 0$. Combining this with (\ref{eq: monotonicity of mu 1}) and the fact that $\Go^{-1}(t)$ is positive definite, implies $\mu_1^\prime(t) < 0$ for all $t \in (0,t_0]$. This completes the proof.
\end{proof}

\subsection{The model ODE}\label{subsec: Model ODE}

We now introduce the model ODE for $\tildeDeltao$ and show that its eigenfunctions are small perturbations of the eigenfunctions of the Airy equation. 

Recall from \Cref{subsec: eigenfcts of Laplace_0 - set up} that $\mu_1(t)$ denotes the first Dirichlet eigenvalue of $-\Div_y(\Go^{-1}(t)\nabla_y)$ in $B_1(0) \subseteq \bbR^{n-1}$. Then, the \emph{model ODE} eigenvalue problem is:
\begin{equation}\label{eq: Model ODE}
	\frac{\jaco(t_0)}{\jaco}\frac{d}{dx}\left(\frac{\jaco}{\jaco(t_0)}\frac{d}{dx}\tilde{h}\right)=
	\left(\delta^{2/3}\varepsilon^{-2}\mu_1-\lambdaODE\right) \tilde{h}
	\quad \text{ on } \big(0,\delta^{-1/3}(t_0+\tau_0) \big)
\end{equation}
with Dirichlet boundary condition. Here, for ease of readability, we simply write $\jaco$ and $\mu_1$ for $\jaco(t_0-\delta^{1/3}x)$ and $\mu_1(t_0-\delta^{1/3}x)$. Recall that $\delta$ was defined in (\ref{eq: def delta}).

The associated ordinary differential operator
\[
	-\frac{\jaco(t_0)}{\jaco}\frac{d}{dx}\left(\frac{\jaco}{\jaco(t_0)}\frac{d}{dx}\right)+\delta^{2/3}\varepsilon^{-2}\mu_1
\] 
is an unbounded self-adjoint operator on
\[
	\tilde{L}^2\big(0,\delta^{-1/3}(t_0+\tau_0)\big) \coloneq \left( L^2\big(0,\delta^{-1/3}(t_0+\tau_0)\big), \langle \cdot, \cdot \rangle_{\tilde{L}^2} \right),
\]
where $\langle \cdot, \cdot \rangle_{\tilde{L}^2}$ is the weighted $L^2$-inner product
\begin{equation}\label{eq: def weighted L^2 norm x}
	\langle w_1, w_2 \rangle_{\tilde{L}^2} \coloneq \int_0^{(t_0+\tau_0)\delta^{-1/3}}\frac{\jaco(t_0-\delta^{1/3}x)}{\jaco(t_0)}w_1(x)w_2(x) \, dx.
\end{equation}
We denote by $(\tilde{h}_i)_{i \in \bbN} \subseteq \tilde{L}^2\big(0,\delta^{-1/3}(t_0+\tau_0)\big)$ a complete orthonormal basis of eigenfunctions of the model ODE (\ref{eq: Model ODE}), and by $\lambdaODE_1 \leq \lambdaODE_2 \leq \dots$ its eigenvalues.

The following is the main result of this subsection. In its formulation, we use the notation from \Cref{subsec: Airy}, i.e., $ 0 > -a_1 > -a_2 > \dots$ denote the zeros of $\Ai$, and $v_k$ ($k \in \bbN$) is the $L^2$-normalized eigenfunctions of the Airy equation on $(0,\infty)$ with eigenvalue $a_k$ defined in (\ref{eq: def of Airy eigfct on half-line}).

\begin{prop}[Approximation of the Model ODE by the Airy equation]\label{prop: Model ODE and Airy}
For all $t_0,\tau_0 > 0$ satisfying (\ref{eq: def tau}) and $k \in \bbN$ there exist $C > 0$ and $\varepsilon_0 > 0$ such that for all $\varepsilon \in (0,\varepsilon_0]$ we have the following estimates:
\begin{enumerate}[(i)]
\item (Eigenvalue)
	\[
		\left|\lambdaODE_k-\delta^{2/3}\varepsilon^{-2}\mu_1(t_0)-a_k \right| \leq C \delta^{1/3}
	\]
\item (Weighted integral)\\
	For every polynomial $Q \in \bbR[X]$
	\[
		\int_0^{(t_0+\tau_0)\delta^{-1/3}}|Q|(x)\left|\tilde{h}_k^2(x)-v_k^2(x) \right| \, dx \leq C_Q \delta^{1/3},
	\]
	where $C_Q$ only depends on the degree of $Q$ and an upper bound for its coefficients (and $t_0$ and $k$). In particular,
	\[
		\left| \int_0^{(t_0+\tau_0)\delta^{-1/3}}x\Big(\tilde{h}_2^2(x)-\tilde{h}_1^2(x) \Big) \, dx - \frac{2}{3}(a_2-a_1) \right| \leq C \delta^{1/3}.
	\]
\item ($H^1$-estimate)
	\[
		\left|\left| \tilde{h}_k \right|\right|_{H^1(0,\delta^{-1/3}(t_0+\tau_0))} \leq C
	\]
\end{enumerate}
\end{prop}

In particular, from \Cref{prop: Model ODE and Airy}(i) we obtain, for all $\varepsilon > 0$ small enough,
\begin{equation}\label{eq: eigenvalue gap ODE}
	\lambdaODE_{k+1}-\lambdaODE_k=a_{k+1}-a_k + O(\delta^{1/3}).
\end{equation}

Here, and from now on, we adopt the following convention regarding constants.

\begin{convention}\label{convention: constants}\normalfont
All implicit constants $C$ are allowed to depend on $t_0, \tau_0 > 0$, $k \in \bbN$, and the underlying metric $g$ (in particular, the local $C^m$-norm of its coefficients), but \emph{not} on $\varepsilon$. Moreover, we will usually suppress this dependence in the notation, e.g., we will just say "there exists a constant $C$" instead of "there exists a constant $C=C(t_0,\tau_0,k,g)$". 
\end{convention}

We start by observing that, up to the addition of a constant, the differential operator underlying the eigenvalue problem (\ref{eq: Model ODE}) is a small perturbation of the Airy operator. This is made precise in the following lemma. In its formulation, $\Ao \coloneq - \frac{d^2}{dx^2}+x$ denotes the Airy operator.

\begin{lem}\label{lem: Model ODE is perturbed Airy}
Define
\[
	\tilde{A} \coloneq 
	-\frac{\jaco(t_0)}{\jaco}\frac{d}{dx}\left(\frac{\jaco}{\jaco(t_0)}\frac{d}{dx}\right)
	+\delta^{2/3}\varepsilon^{-2}\Big(\mu_1\big(t_0-\delta^{1/3}x\big)-\mu_1(t_0) \Big)
\]
and, for all $i \in \bbN$,
\[
	\tilde{\alpha}_i \coloneq \lambdaODE_i-\delta^{2/3}\varepsilon^{-2}\mu_1(t_0).
\]
Then we have 
\(
	\tilde{A}\tilde{h}_i=\tilde{\alpha}_i \tilde{h}_i
\)
for all $i \in \bbN$, and
\[
	\tilde{A}-\Ao=O\big(\delta^{1/3}\big)\frac{d}{dx}+O\big(\delta^{1/3}x^2 \big){\rm id}.
\]
\end{lem}

Observe that with this notation, \Cref{prop: Model ODE and Airy}(i) is equivalent to $|\tilde{\alpha}_k-a_k|\leq C \delta^{1/3}$.

\begin{proof}
The fact that $\tilde{A}\tilde{h}_i=\tilde{\alpha}_i \tilde{h}_i$ for all $i \in \bbN$ is obvious from (\ref{eq: Model ODE}). To estimate $\tilde{A}-A$ note that, by Taylor, 
\[
	\delta^{2/3}\varepsilon^{-2}\Big(\mu_1\big(t_0-\delta^{1/3}x\big)-\mu_1(t_0)\Big)
	=\delta^{2/3}\varepsilon^{-2}\big(-\mu_1^\prime(t_0)\delta^{1/3}x+O\big(\delta^{2/3}x^2 \big) \big)
	=x+O\big(\delta^{1/3}x^2 \big),
\]
where in the second equality we used the definition (\ref{eq: def delta}) of $\delta$. Moreover,
\[
	-\frac{\jaco(t_0)}{\jaco(t_0-\delta^{1/3}x)}\frac{d}{dx}\left(\frac{\jaco(t_0-\delta^{1/3}x)}{\jaco(t_0)}\frac{d}{dx}\right)=-\frac{d^2}{dx^2}+\delta^{1/3}\frac{\jaco^\prime}{\jaco}\frac{d}{dx},
\]
where $\jaco^\prime = \frac{d \jaco}{dt}$. From this the desired estimate on $\tilde{A}-\Ao$ readily follows.
\end{proof}

The proof of \Cref{prop: Model ODE and Airy} is now similar to the proof of  \cite[Proposition 3.3]{CJN25}. Consequently, we only highlight the main points and necessary adjustments, and refer to \cite{CJN25} for further details.

\begin{proof}[Proof of \Cref{prop: Model ODE and Airy}]
Throughout the proof we make use of \Cref{convention: constants} , e.g., if we say "for all $\varepsilon$ small enough" this is supposed to be understood as "for all $\varepsilon$ small enough (depending only on $t_0,\tau_0,k$, and the metric $g$)".

We will also use the following notation. For ease of notation, we set $R \coloneq (t_0+\tau_0)\delta^{-1/3}$. We denote by $(\ho_i)_{i \in \bbN} \subseteq L^2(0,R)$ a complete orthonormal system of eigenfunctions of the Airy operator $\Ao \coloneq -\frac{d^2}{dx^2}+x$ on $(0,R)$ with Dirichlet boundary data, and by $\alphao_1 \leq \alphao_2 \leq \dots$ the respective eigenvalues. Using that the eigenfunctions $v_k$ ($k \in \bbN$) of $\Ao$ on $(0,\infty)$ decay exponentially fast, one can easily show (see for example \cite[Lemma 3.2]{CJN25}) that, after potentially changing the eigenfunctions up to sign, we have, for all $k \in \bbN$,
\begin{equation}\label{eq: Airy from finite to infinite}
	\lim_{R \to \infty}\alphao_k = a_k
	\quad \text{and} \quad
	\lim_{R \to \infty}\left|\left|\ho_k-v_k\left|_{[0,R]}\right.\right|\right|_{L^2(0,R)}=0,
\end{equation}
and that in both cases the convergence is exponentially fast. 

We can now prove \Cref{prop: Model ODE and Airy} in a sequence of claims.

\smallskip\noindent
\textit{Claim 1. There exists $x^\ast >0$ and $C > 0$ such that for all $\varepsilon$ small enough we have
\[
	\left|\ho_k (x)\right|, \left|(\ho_k)^\prime(x) \right| \leq Ce^{-x} 
	\quad \text{ for all } \quad 
	x \in [x^\ast,R].
\]
In particular, for every polynomial $Q \in \bbR[X]$ we have 
\[
	\int_0^{R}|Q|(x) \ho_k(x)^2 \, dx \leq C_Q 
	\quad \text{and} \quad 
	\int_0^{R}|Q|(x) (\ho_k)^\prime(x)^2 \, dx \leq C_Q 
\]
for a finite constant $C_Q$ only depending on the degree of $Q$ and an upper bound for the absolute value of its coefficients (and $t_0,\tau_0,k$).}

In accordance with \Cref{convention: constants}, the constant $x^\ast$ is allowed to depend on $k$.

\smallskip\noindent
\textit{Proof of Claim 1.}
This follows easily from the Sturm comparison theorem. We refer to the proof of \cite[Claim 1 in the proof of Proposition 3.3]{CJN25} for further details.
\hfill$\blacksquare$

\smallskip\noindent
\textit{Claim 2. We have, for all $\varepsilon > 0$ small enough,
\[
	\tilde{\alpha}_k \leq \alphao_k + C \delta^{1/3},
\]
where $\tilde{\alpha}_k$ is as in \Cref{lem: Model ODE is perturbed Airy}.}

\smallskip\noindent
\textit{Proof of Claim 2.}
It follows from the definition (\ref{eq: def weighted L^2 norm x}) of $\langle \cdot, \cdot \rangle_{\tilde{L}^2}$ and from Claim 1 that \cite[(3.2) in Lemma 3.1]{CJN25} is satisfied with $\varepsilon_k=O\big( \delta^{1/3} \big)$. Thus, Claim 2 follows from the upper bound in \cite[(3.4) in Lemma 3.1]{CJN25}, \Cref{lem: Model ODE is perturbed Airy}, and Claim 1. We refer to the proof of \cite[Claim 2 in the proof of Proposition 3.3]{CJN25} for further details.
\hfill$\blacksquare$

\smallskip\noindent
\textit{Claim 3. There exists $x^\ast >0$ and $C > 0$ such that for all $\varepsilon$ small enough we have
\[
	\left|\tilde{h}_k(x)\right|\leq Ce^{-x} 
	\quad \text{ for all } \quad 
	x \in [x^\ast,R]
\]
and
\[
	\int_{x}^R (\tilde{h}_k^\prime)^2(\zeta) \, d\zeta \leq C e^{-x}
	\quad \text{ for all } \quad 
	x \in [x^\ast,R].
\]
In particular, for every polynomial $Q \in \bbR[X]$ we have 
\[
	\int_0^{R}|Q|(x) \tilde{h}_k^2(x) \, dx \leq C_Q 
	\quad \text{and} \quad 
	\int_0^{R}|Q|(x) \big(\tilde{h}_k^\prime(x)\big)^2 \, dx \leq C_Q.
\]}

With a tiny bit more work one could also show $|\tilde{h}_k^\prime(x)| \leq Ce^{-x}$ for all $x \geq x^\ast$.

\smallskip\noindent
\textit{Proof of Claim 3.}
From \Cref{lem: Model ODE is perturbed Airy} we have
\[
	\frac{\jaco(t_0)}{\jaco}\frac{d}{dx}\left(\frac{\jaco}{\jaco(t_0)}\tilde{h}_k^\prime\right)
	=\bigg(\delta^{2/3}\varepsilon^{-2}\Big(\mu_1\big(t_0-\delta^{1/3}x\big)-\mu_1(t_0) \Big) -\tilde{\alpha}_k \bigg)\tilde{h}_k.
\]
Abbreviate $\tilde{c}_k(x) \coloneq \delta^{2/3}\varepsilon^{-2}\Big(\mu_1\big(t_0-\delta^{1/3}x\big)-\mu_1(t_0) \Big) -\tilde{\alpha}_k$. We claim that there exists $x^\ast > 0$ such that $\tilde{c}_k(x) \geq 4$ for all $x \in [x^\ast,R]$. To see this, note that $\tilde{c}_k$ is inreasing on $[0,\delta^{-1/3}t_0]$ due to \Cref{lem: Monotonicity of mu_1}. Moreover, by Taylor (see the proof of \Cref{lem: Model ODE is perturbed Airy}), we also have
\[
	\delta^{2/3}\varepsilon^{-2}\Big(\mu_1\big(t_0-\delta^{1/3}x\big)-\mu_1(t_0) \Big) 
	=x+O\big(\delta^{1/3}x^2\big).
\]
Finally, it follows from (\ref{eq: Airy from finite to infinite}) and Claim 2 that $\tilde{\alpha}_k \leq a_k+1$ for all $\varepsilon > 0$ small enough. Therefore, if we set $x^\ast \coloneq a_k+10$, then for all $x \in [x^\ast,\delta^{-1/3}t_0]$ and all $\varepsilon > 0$ small enough
\[
	\tilde{c}_k(x) \geq \tilde{c}_k(x^\ast)=x^\ast+O\big(\delta^{1/3}(x^\ast)^2\big)-\tilde{\alpha}_k \geq 4.
\]
Note that, by symmetry, \Cref{lem: Monotonicity of mu_1} also gives $\mu_1^\prime(t) > 0$ for $t \in [-\tau_0,0)$. Hence, by our choice (\ref{eq: def tau}) of $\tau_0$, we also have $\mu_1(t) \geq \mu_1(t_0/4)$ for all $t \in [-\tau_0,0)$. From this one can easily deduce that $\tilde{c}_k(x) \geq 4$ also holds for all $x \in [\delta^{-1/3}t_0,R]$. This completes the proof that $\tilde{c}_k(x) \geq 4$ for all $x \in [x^\ast,R]$.

Therefore,
\[
	\tilde{h}_k^{\prime \prime}(x)=\tilde{c}_k(x)\tilde{h}_k(x)+\delta^{1/3}b(x)\tilde{h}_k^\prime(x)
\]
for a continuous function $b$ with $||b||_{C^0}=O(1)$, and $\tilde{c}_k(x) \geq 4$ for $x \in [x^\ast,R]$. Exploiting the chain of inequalities, valid for all $\varepsilon > 0$ small enough,
\begin{align*}
	\left(\frac{1}{2}\tilde{h}_k^2	 \right)^{\prime \prime}
	=\big(\tilde{h}_k \tilde{h}_k^\prime\big)^{\prime}
	=\left(\tilde{h}_k^\prime \right)^2+\tilde{h}_k \big(\tilde{c}_k\tilde{h}_k+\delta^{1/3}b\tilde{h}_k^\prime\big)
	\geq \frac{1}{2}\left(\tilde{h}_k^\prime \right)^2
	+\frac{\tilde{c}_k}{2}\tilde{h}_k^2
	\geq \sqrt{\tilde{c}_k}\big| \tilde{h}_k \tilde{h}_k^\prime \big|
\end{align*}
one can easily deduce the desired estimates. Indeed, we first observe that $\tilde{h}_k$ and $|\tilde{h}_k \tilde{h}_k^\prime|$ are exponentially decaying in $[x^\ast,R]$ because $(\frac{1}{2}\tilde{h}_k^2)^{\prime \prime} \geq \frac{\tilde{c}_k}{2}\tilde{h}_k^2$ and $(\tilde{h}_k\tilde{h}_k^\prime)^\prime \geq \sqrt{\tilde{c}_k} |\tilde{h}_k \tilde{h}_k^\prime |$. Consequently, integrating $(\tilde{h}_k\tilde{h}_k^\prime)^\prime \geq \frac{1}{2} (\tilde{h}_k^\prime)^2$ yields that $\int_x^R (\tilde{h}_k^\prime)^2(\zeta) d\zeta$ is exponentially decaying. We refer to the proof of \cite[Claim 3 in the proof of Proposition 3.3 in Section 4.2]{CJN25} for further details.
\hfill$\blacksquare$

\smallskip\noindent
\textit{Claim 4. We have, for all $\varepsilon > 0$ small enough,
\[
	|\tilde{\alpha}_k - \alphao_k | \leq C \delta^{1/3},
\]
where $\tilde{\alpha}_k$ is as in \Cref{lem: Model ODE is perturbed Airy}.}

\smallskip\noindent
\textit{Proof of Claim 4.}
Using Claim 3 instead of Claim 1, this follows from \cite[(3.4) in Lemma 3.1]{CJN25} and \Cref{lem: Model ODE is perturbed Airy} exactly as in the proof of Claim 2.
\hfill$\blacksquare$

\smallskip\noindent
\textit{Claim 5. After potentially changing $\tilde{h}_k$ up to sign, we have, for all $\varepsilon > 0$ small enough,
\[
	||\tilde{h}_k - \ho_k ||_{L^2(0,R)} \leq C \delta^{1/3}.
\]}

\smallskip\noindent
\textit{Proof of Claim 5.}
This immediately follows from \Cref{lem: Model ODE is perturbed Airy}, Claim 1, Claim 3, and \cite[(3.5) in Lemma 3.1]{CJN25}. We refer to the proof of \cite[Claim 5 in the proof of Proposition 3.3]{CJN25} for further details.
\hfill$\blacksquare$

With Claims 1-5 at hand, we can now easily complete the proof of \Cref{prop: Model ODE and Airy}.

Recall that $|\alphao_k-a_k|$ is exponentially small in $R=\delta^{-1/3}(t_0+\tau_0)$ by (\ref{eq: Airy from finite to infinite}). Thus, keeping in mind the definition of $\tilde{\alpha}_k$ (see \Cref{lem: Model ODE is perturbed Airy}), \Cref{prop: Model ODE and Airy}(i) follows immediately from Claim 4.

Similarly, $||\ho_k-v_k||_{L^2(0,R)}$ is exponentially small in $R=\delta^{-1/3}(t_0+\tau_0)$ by (\ref{eq: Airy from finite to infinite}), and thus $||\tilde{h}_k-v_k||_{L^2(0,R)}=O(\delta^{1/3})$ due to Claim 5. For any polynomial $Q \in \bbR[X]$ we have
\[
    \int_0^{R}|Q|(x) \left|\tilde{h}_k^2(x)-v_k^2(x)\right| \, dx
    \leq \big(\left|\left|Q\tilde{h}_k \right|\right|_{L^2}
    +\left|\left|Qv_k \right|\right|_{L^2}\big)
    \left|\left|\tilde{h}_k-v_k \right|\right|_{L^2}
\]
due to the third binomial formula, the Cauchy-Schwarz inequality, and the triangle inequality. Note that Claim 1 also holds for $v_k$ (with the same proof).  Thus the first estimate in \Cref{prop: Model ODE and Airy}(ii) follows from Claim 1 for $v_k$, Claim 3, and the unweighted estimate $||\tilde{h}_k-v_k||_{L^2}=O(\delta^{1/3})$. The second estimate in \Cref{prop: Model ODE and Airy}(ii) follows from the first and \Cref{lem: model integral}.

Finally, \Cref{prop: Model ODE and Airy}(iii) is the special case of Claim 3 with $Q \equiv 1$.
\end{proof}

\subsection{Approximate separation of variables}\label{subsec: Separation of variables}

The goal of this subsection is to show that, up to small errors, the eigenfunctions of $\tildeDeltao$ can essentially be obtained from the eigenfunctions of the model ODE (\ref{eq: Model ODE}) by a separation of variables ansatz.

To formulate the result, we continue to use the notation introduced in \Cref{subsec: eigenfcts of Laplace_0 - set up}, and recall that $f_1^{(t)}$ ($t \in [-\tau_0,t_0]$) denotes the first $L^2$-normalized Dirichlet eigenfunction of $-\Div_y\big(\Go^{-1}(t)\nabla_y \big)$ in $B_1(0) \subseteq \bbR^{n-1}$. In accordance with the change of variables (\ref{eq: change of variables}), we define
\[
	\tilde{f}_1^{(x)} \coloneq f_1^{(t_0-\delta^{1/3}x)}.
\]
Observe that, for all $x \in [0,\delta^{-1/3}(t_0+\tau_0)]$, we have
\begin{equation}\label{eq: bounds for f_1}
	\left|\left|\partial_x \tilde{f}_1^{(x)}\right|\right|_{L^2(B_1)} \leq C \delta^{1/3}
	\quad \text{and} \quad 
	\left|\left|\partial^2_{xx} \tilde{f}_1^{(x)}\right|\right|_{L^2(B_1)} \leq C \delta^{2/3}
\end{equation}
for a constant $C$ only depending on $t_0,\tau_0$ and $(X,g)$. Indeed, this just follows from the fact that $t \mapsto f_1^{(t)}$ is smooth as a $L^2$-valued map.

Recall that $(\tilde{h}_i)_{i \in \bbN} \subseteq \tilde{L}^2\big(0,\delta^{-1/3}(t_0+\tau_0) \big)$ denotes a complete orthonormal basis of eigenfunctions of the model ODE (\ref{eq: Model ODE}), and $\tilde{L}^2$ is defined by the weighted $L^2$-norm (\ref{eq: def weighted L^2 norm x}).

We define $\uGuess_i \in H^2(\tildeOmega) \cap H_0^1(\tildeOmega)$ for all $i \in \bbN$ by
\begin{equation}\label{eq: u Guess}
	\uGuess_i \, \colon \tildeOmega \longrightarrow \bbR, \, 
	(x,y) \longmapsto \tilde{h}_i(x)\tilde{f}_1^{(x)}(y),
\end{equation}
where $\tildeOmega \subseteq \bbR \times \bbR^{n-1}$ is the region defined in (\ref{eq: def tilde(Omega)}).

Finally, we denote by $(\uo_i)_{i \in \bbN} \subseteq \tilde{L}^2_\circ(\tildeOmega)$ a complete orthonormal basis of eigenfunctions of $-\tildeDeltao$ in $\tildeOmega$ with Dirichlet boundary data, and by $\lambdao_1 \leq \lambdao_2 \leq \dots$ its eigenvalues. Here $\tilde{L}^2_\circ(\tildeOmega)$ denotes the weighted $L^2$-space defined in (\ref{eq: def weighted L_0^2 norm xy}).

We can now formulate the main result of this subsection.

\begin{prop}[Approximate Separation of Variables]\label{prop: Separation of variables}
For all $t_0,\tau_0 > 0$ satisfying (\ref{eq: def tau}) and $k \in \bbN$ there exist $C > 0$ and $\varepsilon_0 > 0$ such that for all $\varepsilon \in (0,\varepsilon_0]$ we have the following estimates:
\begin{enumerate}[(i)]
	\item (Eigenvalue)
	\[
		\left|\lambdao_k - \lambdaODE_k \right| \leq C \delta^{1/3}
	\]
	\item (Eigenfunction)
	\[
		\left|\left|\uo_k - \uGuess_k \right|\right|_{\tilde{L}^2_\circ(\tildeOmega)} \leq C\delta^{1/3}
	\]
\end{enumerate}
\end{prop}

The strategy for \Cref{prop: Separation of variables} is as follows. First, we show that every $\uGuess_i$ is an approximate eigenfunction. From this we can deduce that every $\lambdaODE_i$ is close to some $\lambdao_{j_i}$. Then we show that any eigenfunction $\uo_j$ induces an approximate eigenfunction $\hGuess_j$ of the model ODE (\ref{eq: Model ODE}), and that the $\hGuess_j$ ($j \in \bbN$) are approximately orthonormal. For the latter, we need to show that "most of" $\uo_j(x,\cdot)$ projects onto $\tilde{f}_1^{(x)}$ in $L^2(B_1)$ (the idea for this is explained after \Cref{lem: PDE to ODE 2}). From this we can deduce \Cref{prop: Separation of variables}.

We now properly implement this strategy in a sequence of lemmas. Throughout, we continue to make use of \Cref{convention: constants} and suppress the dependence of $t_0,\tau_0$ in the formulation of the following lemmas.

\begin{lem}\label{lem: ODE to PDE}
For every $k \in \bbN$ there exists a constant $\varepsilon_0 > 0$ with the following property.
For all $\varepsilon \in (0, \varepsilon_0]$ and $i=1,\dots,k$ we have
\[
	\left|\left|(-\tildeDeltao -\lambdaODE_i)\uGuess_i \right| \right|_{\tilde{L}^2_\circ(\tildeOmega)}=O\big( \delta^{1/3}\big)
	\quad \text{and} \quad
	\left|\left|\uGuess_i \right| \right|_{\tilde{L}^2_\circ(\tildeOmega)} = 1.
\]
In particular, for all $i=1,\dots,k$, there exists an eigenvalue $\lambdao_{j_i}$ of $-\tildeDeltao$ with 
\[
	\left| \lambdaODE_i-\lambdao_{j_i} \right| = O\big( \delta^{1/3}\big).
\]
Moreover, $j_1 <  \dots < j_k$.
\end{lem}

Indeed, this just follows from an easy application of \Cref{lem: Guessing eigenvectors}.

\begin{proof}
Since  $||\tilde{h}_i||_{\tilde{L}^2(0,\delta^{-1/3}(t_0+\tau_0))}=1$, and $\big|\big| \tilde{f}^{(x)} \big|\big|_{L^2(B_1)}=1$ for all $x \in [0,\delta^{-1/3}(t_0+\tau_0)]$, we easily obtain 
\[
	\big|\big| \uGuess_i \big|\big|_{\tilde{L}^2_\circ(\tildeOmega)}=1.
\]
So it suffices to bound $(-\tildeDeltao -\lambdaODE_i)\uGuess_i$. To do so we compute
\begin{align*}
	\partial_x\left(\jaco \partial_x \uGuess_i \right) =& \frac{d}{dx}\left(\jaco \tilde{h}^\prime_i \right)\tilde{f}^{(x)} \\
	&+2\jaco \tilde{h}_i^\prime \partial_x \tilde{f}^{(x)}+\left(\frac{d}{dx} \jaco \right)\tilde{h}_i \partial_x \tilde{f}^{(x)}+\jaco \tilde{h}_i \partial^2_{xx}\tilde{f}^{(x)}.
\end{align*}
Hence, by (\ref{eq: def tilde(Laplace)_0}),
\begin{align*}
	\tildeDeltao \uGuess_i
	=& \frac{1}{\jaco}\partial_x\left(\jaco \partial_x  \uGuess_i\right)
	+\delta^{2/3}\varepsilon^{-2} \Div_y\left(\Go^{-1}(t_0-\delta^{1/3}x) \nabla_y \uGuess_i\right) \\
	=& \left( \frac{1}{\jac(t_0)}\frac{d}{dx}\left(\jaco \tilde{h}^\prime_i \right) \right)\tilde{f}^{(x)}+\tilde{h}_i \delta^{2/3}\varepsilon^{-2} \Div_y\left(\Go^{-1}(t_0-\delta^{1/3}x) \nabla_y \tilde{f}_1^{(x)}\right) \\
	&+ 2\tilde{h}_i^\prime \partial_x \tilde{f}^{(x)}+\frac{1}{\jaco}\left(\frac{d}{dx} \jaco \right)\tilde{h}_i \partial_x \tilde{f}^{(x)}+ \tilde{h}_i \partial^2_{xx}\tilde{f}^{(x)}.
\end{align*}
Keeping in mind that, by definition of $\tilde{f}_1^{(x)}$, 
\[
	-\Div_y\left(\Go^{-1}(t_0-\delta^{1/3}x) \nabla_y \tilde{f}_1^{(x)}\right)=\mu_1(t_0-\delta^{1/3}x)\tilde{f}_1^{(x)},
\]
and that, from the eigenvalue problem (\ref{eq: Model ODE}) for $\tilde{h}_i$,
\[
	\frac{1}{\jaco}\frac{d}{dx}\left(\jaco\frac{d}{dx}\tilde{h}_i\right)=
	\left(\delta^{2/3}\varepsilon^{-2}\mu_1-\lambdaODE_i\right) \tilde{h}_i,
\]
we obtain
\begin{align*}
	\tildeDeltao \uGuess_i
=&-\lambdaODE_i \uGuess_i 
	+ 2\tilde{h}_i^\prime \partial_x \tilde{f}^{(x)}+\frac{1}{\jaco}\left(\frac{d}{dx} \jaco \right)\tilde{h}_i \partial_x \tilde{f}^{(x)}+ \tilde{h}_i \partial^2_{xx}\tilde{f}^{(x)}.
\end{align*}
Therefore, the desired estimate on $(-\tildeDeltao-\lambdaODE_i)\uGuess_i$ follows from \Cref{prop: Model ODE and Airy}(iii), $\left|\frac{d}{dx} \jaco \right|=O\big(\delta^{1/3} \big)$, and (\ref{eq: bounds for f_1}). 

The fact that every $\lambdaODE_i$ is $O(\delta^{1/3})$-close to some eigenvalue $\lambdao_{j_i}$ of $-\tildeDeltao$ then follows immediately from (\ref{eq: guessing eigenvalues}) in \Cref{lem: Guessing eigenvectors}. Finally, $j_i < j_{i+1}$ for all $i=1,\dots,k-1$ follows from the fact that, thanks to (\ref{eq: eigenvalue gap ODE}),
\[
	\lambdaODE_{i+1}-\lambdaODE_i = a_{i+1}-a_i+O(\delta^{1/3}) \geq (a_{i+1}-a_i)/2 > 0
\] 
is uniformly bounded from below for all $\varepsilon > 0$ small enough. 
\end{proof}

In particular, we easily obtain the following upper bound.

\begin{cor}\label{cor: upper bound lambda(Delta_0)}
For all $k \in \bbN$ there exists $\varepsilon_0 > 0$ such that for all $\varepsilon \in (0, \varepsilon_0]$ we have
\[
	\lambdao_k \leq \lambdaODE_k + O\big(\delta^{1/3} \big).
\]
\end{cor}

\begin{proof}
Using the notation from \Cref{lem: ODE to PDE}, note $k \leq j_k$ since $j_1 < \dots < j_k$, and thus $\lambdao_k \leq \lambdao_{j_k}$. So the desired upper bound is immediate from \Cref{lem: ODE to PDE}.
\end{proof}

On its own, the bound in \Cref{lem: ODE to PDE} is too weak to deduce \Cref{prop: Separation of variables}(i). Indeed, a priori it could be the case that several eigenvalues of $-\tildeDeltao$ accumulate around each eigenvalue $\lambdaODE_i$ of the model ODE (\ref{eq: Model ODE}). To exclude this behaviour, we also need to show that every eigenfunction $\uo_j$ of $-\tildeDeltao$ induces an approximate eigenfunction $\hGuess_j$ of the model ODE (\ref{eq: Model ODE}), and that the $\hGuess_j$ ($j \in \bbN$) are almost orthogonal. \Cref{prop: Separation of variables} will then follow from \Cref{lem: Guessing eigenvectors}.

We define $\hGuess_j$ for all $j \in \bbN$ by
\begin{equation}\label{eq: h Guess}
	\hGuess_j \, \colon \big(0,\delta^{-1/3}(t_0+\tau_0) \big) \to \bbR, \,
	x \mapsto 
	\left\langle \uo_j(x, \cdot), \tilde{f}_1^{(x)}(\cdot) \right\rangle_{L^2(B_1)}.
\end{equation}
We start with the following observation. In its formulation we continue to simply write $\jaco$ and $\mu_1$ instead of $\jaco \big(t_0-\delta^{1/3}x \big)$ and $\mu_1\big(t_0-\delta^{1/3}x \big)$ for ease of readability. We also recall that $\tilde{L}^2\big(0,\delta^{-1/3}(t_0+\tau_0)\big)$ denotes the weighted $L^2$-space defined in (\ref{eq: def weighted L^2 norm x}).

\begin{lem}\label{lem: PDE to ODE 1}
For every $k \in \bbN$ there exists a constant $\varepsilon_0 > 0$ with the following property.
For all $\varepsilon \in (0, \varepsilon_0]$ and $j=1,\dots,k$, we have
\[
	\left|\left|\left(-\frac{1}{\jaco}\frac{d}{dx}\left(\jaco \frac{d}{dx}\right)+\delta^{2/3}\varepsilon^{-2}\mu_1 - \lambdao_j \right)\hGuess_j \right|\right|_{\tilde{L}^2(0,\delta^{-1/3}(t_0+\tau_0))}=O\big(\delta^{1/6}\big).
\]
\end{lem}

With a bit more work, it is possible to improve the upper bound to $O\big(\delta^{1/3}\big)$. Since this is not needed for our purposes, we refrain from doing so.

To deduce that $\hGuess_j$ is an almost-eigenfunction of (\ref{eq: Model ODE}), we also need that $||\hGuess_j||_{\tilde{L}^2}$ is uniformly bounded from below. This will be shown in \Cref{lem: PDE to ODE 2} below.

\begin{proof}
Using the definition (\ref{eq: h Guess}) of $\hGuess_j$, a straightforward computation yields
\begin{align*}
	\frac{1}{\jaco}\frac{d}{dx}\left(\jaco \frac{d}{dx}\hGuess_j\right)(x)
	=& \left\langle \frac{1}{\jaco}\partial_x\big(\jaco \partial_x\uo_j\big)(x,\cdot), \tilde{f}_1^{(x)}(\cdot) \right\rangle_{L^2(B_1)} \\
	&+2\left\langle \partial_x \uo_j(x, \cdot), \partial_x\tilde{f}_1^{(x)}(\cdot) \right\rangle_{L^2(B_1)} \\
	&+\frac{1}{\jaco}\frac{d \jaco}{dx}\left\langle \uo_j(x, \cdot), \partial_x\tilde{f}_1^{(x)}(\cdot) \right\rangle_{L^2(B_1)} \\
	&+ \left\langle \uo_j(x, \cdot), \partial^2_{xx}\tilde{f}_1^{(x)}(\cdot) \right\rangle_{L^2(B_1)}.
\end{align*}
Keeping in mind $-\tildeDeltao \uo_j = \lambdao_j \uo_j$ and the formula (\ref{eq: def tilde(Laplace)_0}) for $-\tildeDeltao$, we can simplify the first summand on the right hand side to
\begin{align*}
	\left\langle \frac{1}{\jaco}\partial_x\big(\jaco \partial_x\uo_j\big)(x,\cdot), \tilde{f}_1^{(x)}(\cdot) \right\rangle_{L^2(B_1)}
	=& \left\langle -\lambdao_j\uo_j(x, \cdot), \tilde{f}_1^{(x)}(\cdot) \right\rangle_{L^2(B_1)} \\
	+&\left\langle -\delta^{2/3}\varepsilon^{-2}\Div_y\left(\Go^{-1}\nabla_y \uo_j (x, \cdot)\right), \tilde{f}_1^{(x)}(\cdot) \right\rangle_{L^2(B_1)} \\
	=& -\lambdao_j \hGuess_j \\
	+&\delta^{2/3}\varepsilon^{-2}\left\langle  \uo_j(x, \cdot), -\Div_y\left(\Go^{-1}\nabla_y\tilde{f}_1^{(x)}(\cdot) \right)\right\rangle_{L^2(B_1)} \\
	=& -\lambdao_j \hGuess_j + \delta^{2/3}\varepsilon^{-2}\left\langle \uo_j(x, \cdot), \mu_1 \tilde{f}_1^{(x)}(\cdot) \right\rangle_{L^2(B_1)} \\
	=& \big(-\lambdao_j+\delta^{2/3}\varepsilon^{-2}\mu_1 \big) \hGuess_j.
\end{align*}
Together with $\left|\frac{d \jaco}{dx} \right|=O\big(\delta^{1/3}\big)$ and (\ref{eq: bounds for f_1}), we obtain
\begin{align*}
	\bigg|\bigg(-\frac{1}{\jaco}\frac{d}{dx}\left(\jaco \frac{d}{dx}\right)&+\delta^{2/3}\varepsilon^{-2}\mu_1 - \lambdao_j \bigg)\hGuess_j \bigg|(x) \\
	& \leq C\delta^{1/3}||\partial_x \uo_j(x, \cdot)||_{L^2(B_1)}
	+ C\delta^{2/3}||\uo_j(x, \cdot)||_{L^2(B_1)}.
\end{align*}
This implies the desired estimate if we can show
\begin{equation}\label{eq: L^2 bound partial_x tilde(u)_0 - weak}
	||\partial_x \uo_j||_{\tilde{L}^2_\circ(\tildeOmega)}=O\big(\delta^{-1/6}\big).
\end{equation}
To see this, test $-\tildeDeltao \uo_j=\lambdao_j \uo_j$ against $\frac{\jaco}{\jaco(t_0)}\uo_j \in H_0^1(\tilde{\Omega})$ to obtain
\[
	\int_{\tildeOmega}\frac{\jaco}{\jaco(t_0)}\left(|\partial_x \uo_j|^2+ 
	\delta^{2/3}\varepsilon^{-2}\left\langle \Go^{-1}\nabla_y \uo_j, \nabla_y \uo_j \right\rangle_{\bbR^{n-1}}\right) \, dx dy = \lambdao_j \int_{\tildeOmega}\frac{\jaco}{\jaco(t_0)} |\uo_j|^2 \, dx dy = \lambdao_j.
\]
Moreover, invoking \Cref{cor: upper bound lambda(Delta_0)}, \Cref{prop: Model ODE and Airy}(i), and the definition (\ref{eq: def delta}) of $\delta$, we can bound
\[
	\lambdao_j  \leq \lambdaODE_j + O\big(\delta^{1/3}\big)=\delta^{2/3}\varepsilon^{-2}\mu_1(t_0)+a_j+O\big(\delta^{1/3}\big)=O\big(\delta^{-1/3}\big).
\]
This establishes (\ref{eq: L^2 bound partial_x tilde(u)_0 - weak}), and thus completes the proof.
\end{proof}

Next we show that the $(\hGuess_j)_{j \in \bbN}$ are almost-orthonormal. This reflects the fact that "most of" $\uo_j(x, \cdot)$ projects onto $\tilde{f}_1^{(x)}$ in $L^2(B_1)$.

\begin{lem}\label{lem: PDE to ODE 2}
For every $k \in \bbN$ there exists a constant $\varepsilon_0 > 0$ with the following property.
For all $\varepsilon \in (0, \varepsilon_0]$ and $i,j \in \{1,\dots,k\}$, we have
\[
	\left|\left\langle \hGuess_i, \hGuess_j \right\rangle_{\tilde{L}^2}-\delta_{ij} \right|=O\big( \delta^{1/3} \big),
\]
where $\langle \cdot, \cdot \rangle_{\tilde{L}^2}$ is the weighted $L^2$-inner product defined in (\ref{eq: def weighted L^2 norm x}).
\end{lem}

The idea is as follows. If this were not true, i.e., if a substantial part of $\uo_j(x,\cdot)$ were to project onto $\big(\tilde{f}_1^{(x)}\big)^\perp$, then, due to the expression (\ref{eq: def tilde(Laplace)_0}) for $\tildeDeltao$, the eigenvalue $\lambdao_j$ would have to be significantly larger than $\delta^{2/3}\varepsilon^{-2}\mu_1(t_0)$. However, this would contradict \Cref{cor: upper bound lambda(Delta_0)} and \Cref{prop: Model ODE and Airy}(i).

\begin{proof}
Decompose $\uo_j=(\uo_j)^\top+(\uo_j)^\perp$, where $(\uo_j)^\top,(\uo_j)^\perp$ are defined such that, for all $x \in [0,\delta^{-1/3}(t_0+\tau_0)]$, we have
\[
	\uo_j(x,\cdot)=(\uo_j)^\top(x, \cdot)+(\uo_j)^\perp(x, \cdot) \in {\rm span}\left\{ \tilde{f}_1^{(x)} \right\}+{\rm span}\left\{ \tilde{f}_1^{(x)} \right\}^\perp = L^2(B_1).
\]
Then 
\[
	\left\langle \hGuess_i, \hGuess_j \right\rangle_{\tilde{L}^2(0,\delta^{-1/3}(t_0+\tau_0))}
	= \left\langle (\uo_i)^\top, (\uo_j)^\top \right\rangle_{\tilde{L}_\circ^2(\tildeOmega)}.
\]
Since
\[
	\delta_{ij}=\left\langle \uo_i, \uo_j\right\rangle_{\tilde{L}_\circ^2(\tildeOmega)}
	=\left\langle (\uo_i)^\top, (\uo_j)^\top \right\rangle_{\tilde{L}_\circ^2(\tildeOmega)}
	+\left\langle (\uo_i)^\perp, (\uo_j)^\perp \right\rangle_{\tilde{L}_\circ^2(\tildeOmega)},
\]
it therefore suffices to show that, for all $j=1,\dots,k$, we have
\begin{equation}\label{eq: normal part of tilde(u)_0 is small}
	\big|\big|(\uo_j)^\perp \big|\big|_{\tilde{L}_\circ^2(\tildeOmega)}^2=O\big(\delta^{1/3}\big).
\end{equation}
Towards (\ref{eq: normal part of tilde(u)_0 is small}), we again test $\lambdao_j \uo_j=-\tildeDeltao \uo_j$ against $\frac{\jaco}{\jaco(t_0)}\uo_j$ and remember $||\uo_j||_{\tilde{L}_\circ^2(\tildeOmega)}=1$ to deduce
\begin{align}\label{eq: normal part is small 1}
	\lambdao_j =& \int_{\tildeOmega}\frac{\jaco}{\jaco(t_0)}\left(|\partial_x \uo_j|^2+ 
	\delta^{2/3}\varepsilon^{-2}\left\langle \Go^{-1}\nabla_y \uo_j, \nabla_y \uo_j \right\rangle_{\bbR^{n-1}}\right) \, dx dy \notag \\
	\geq & \int_{0}^{(t_0+\tau_0)\delta^{-1/3}}\frac{\jaco}{\jaco(t_0)} 
	\delta^{2/3}\varepsilon^{-2}\left\langle \Go^{-1}\nabla_y \uo_j (x,\cdot), \nabla_y \uo_j (x,\cdot)\right\rangle_{L^2(B_1)} \, dx.
\end{align}
Denote by $\mu_2(t)$ the second Dirichlet eigenvalue of $-\Div\left(\Go^{-1}(t)\nabla \right)$ in $B_1(0) \subseteq \bbR^{n-1}$. Then 
\[
	\left\langle \Go^{-1}\nabla_y \uo_j (x,\cdot), \nabla_y \uo_j (x,\cdot)\right\rangle_{L^2(B_1)} \geq  \mu_1 ||(\uo_j)^\top(x,\cdot)||_{L^2(B_1)}^2 + \mu_2 ||(\uo_j)^\perp(x,\cdot)||_{L^2(B_1)}^2.
\]
There exists $\mu_{\rm gap}=\mu_{\rm gap}(t_0,\tau_0) > 0$ such that $\mu_2(t) \geq \mu_1(t)+\mu_{\rm gap} $ for all $t \in [-\tau_0,t_0]$. Indeed, this follows from the fact that $\mu_1(t)$ is simple, and by the compactness of $[-\tau_0,t_0]$.
Thus
\begin{equation*}\label{eq: normal part is small 2}
	\left\langle \Go^{-1}\nabla_y \uo_j (x,\cdot), \nabla_y \uo_j (x,\cdot)\right\rangle_{L^2(B_1)} \geq  \mu_1 ||(\uo_j)(x,\cdot)||_{L^2(B_1)}^2 + \mu_{\rm gap} ||(\uo_j)^\perp(x,\cdot)||_{L^2(B_1)}^2.
\end{equation*}
We know $\mu_1(t) \geq \mu_1(t_0)$ for all $t \in [-\tau_0,t_0]$ due to \Cref{lem: Monotonicity of mu_1} and the choice (\ref{eq: def tau}) of $\tau_0$. Plugging this into (\ref{eq: normal part is small 1}), and keeping in mind $||\uo_j||_{\tilde{L}_\circ^2(\tildeOmega)}=1$, yields
\begin{equation}\label{eq: normal part is small 3}
	\lambdao_j \geq 
	\delta^{2/3}\varepsilon^{-2}\mu_1(t_0)
	+\delta^{2/3}\varepsilon^{-2}\mu_{\rm gap}\left|\left|(\uo_j)^\perp \right|\right|_{\tilde{L}_\circ^2(\tildeOmega)}^2.
\end{equation}
However, we also know
\begin{equation}\label{eq: normal part is small 4}
	\lambdao_j \leq \lambdaODE_j + O\big(\delta^{1/3}\big)=\delta^{2/3}\varepsilon^{-2}\mu_1(t_0)+a_j +  O\big(\delta^{1/3}\big)
\end{equation}
thanks to \Cref{cor: upper bound lambda(Delta_0)} and \Cref{prop: Model ODE and Airy}(i). Finally, we have $\delta^{-2/3}\varepsilon^{2}= O\big(\delta^{1/3}\big)$ due to the definition (\ref{eq: def delta}) of $\delta$. Therefore, combining (\ref{eq: normal part is small 3}) and (\ref{eq: normal part is small 4}) establishes (\ref{eq: normal part of tilde(u)_0 is small}). This completes the proof.
\end{proof}

With these lemmas at hand, we can now easily deduce \Cref{prop: Separation of variables}.

\begin{proof}[Proof of \Cref{prop: Separation of variables}]
Thanks to \Cref{lem: PDE to ODE 1} and \Cref{lem: PDE to ODE 2} we can apply \Cref{lem: Guessing eigenvectors} to the differential operator $-\frac{1}{\jaco}\frac{d}{dx}\left(\jaco \frac{d}{dx} \right)+\delta^{2/3}\varepsilon^{-2}\mu_1$ from the model ODE (\ref{eq: Model ODE}) with $\varepsilon_{\rm Guess}=O\big( \delta^{1/6}\big)$. Therefore, for all $j=1,\dots,k$ there exists $i_j \in \bbN$ such that 
\[
	\left|\lambdao_j - \lambdaODE_{i_j}\right|=O\big( \delta^{1/6}\big)
	\quad \text{and} \quad
	\left|\left|\hGuess_j - \tilde{h}_{i_j} \right|\right|_{\tilde{L}^2}=O\big( \delta^{1/6}\big),
\]
where for the second estimate we also used that, for $\varepsilon > 0$ small enough, the eigenvalue gaps $\lambdaODE_{i+1}-\lambdaODE_i$ are uniformly bounded from below thanks to (\ref{eq: eigenvalue gap ODE}).

We claim that $i_j=j$ for all $j=1,\dots,k$. Indeed, from \Cref{cor: upper bound lambda(Delta_0)} we know
\[
	\lambdaODE_{i_j} = \lambdao_j + O\big( \delta^{1/6}\big)
	\leq \lambdaODE_j +O\big( \delta^{1/6}\big),
\]
and thus $i_j \leq j$ by (\ref{eq: eigenvalue gap ODE}). Moreover, we also know $||\hGuess_j-\hGuess_{j^\prime}||_{\tilde{L}^2}^2=2+O(\delta^{1/3})$ for all $j \neq j^\prime$ due to \Cref{lem: PDE to ODE 2}. This implies that $j \mapsto i_j$ ($j=1,\dots,k$) is injective. Combining these two observations shows $i_j=j$ for all $j=1,\dots,k$, as claimed.

Until now we have shown $|\lambdao_j - \lambdaODE_{j}|=O\big( \delta^{1/6}\big)$ for all $j=1,\dots,k$. Together with the eigenvalue gap (\ref{eq: eigenvalue gap ODE}), and using the notation from \Cref{lem: ODE to PDE}, this implies $j_i=i$.  Thus,  applying \Cref{lem: Guessing eigenvectors} to \Cref{lem: ODE to PDE} once more, we obtain, for all $i=1,\dots,k$,
\[
	\left|\lambdaODE_i - \lambdao_{i}\right|=O\big( \delta^{1/3}\big)
	\quad \text{and} \quad
	\left|\left|\uGuess_i - \uo_i \right|\right|_{\tilde{L}_\circ^2(\tildeOmega)}=O\big( \delta^{1/3}\big),
\]
where we used that, as a consequence of  $|\lambdao_j - \lambdaODE_{j}|=O\big( \delta^{1/6}\big)$ and (\ref{eq: eigenvalue gap ODE}), the eigenvalue gaps $\lambdao_{i+1}-\lambdao_i$ of $-\tildeDeltao$ are also uniformly bounded from below for all $\varepsilon > 0$ small enough. This completes the proof.
\end{proof}

\section{Eigenfunctions of \(\tildeDelta\)}\label{sec: Laplace}

We now finally turn our attention to the eigenfunctions of $\tildeDelta$. Namely, we will show in \Cref{subsec: from Laplace_0 to Laplace} that the eigenfunctions of $\tildeDelta$ are small perturbations of the eigenfunctions of $\tildeDeltao$. Together with the results of \Cref{sec: Laplace_0}, we will then be able to prove \Cref{Main Theorem}. This will be carried out in \Cref{subsec: Proof of Main Thm}.

\subsection{Approximation of \(\tildeDelta\) by \(\tildeDeltao\)}\label{subsec: from Laplace_0 to Laplace}

We denote by $(\tilde{u}_i)_{i \in \bbN} \subseteq \tilde{L}^2(\tildeOmega)$ a complete orthonormal basis of eigenfunctions of $-\tildeDelta$ in $\tildeOmega$ with Dirichlet boundary data, and by $\tilde{\lambda}_1 \leq \tilde{\lambda}_2 \leq \dots$ its eigenvalues. Here $\tilde{L}^2(\tildeOmega)$ denotes the weighted $L^2$-space defined in (\ref{eq: def weighted L^2 norm xy}). We recall from \Cref{subsec: Separation of variables} that $(\uo_i)_{i \in \bbN} \subseteq \tilde{L}^2_\circ(\tildeOmega)$ and $(\lambdao_i)_{i \in \bbN}$ denote the Dirichlet eigenfunctions and eigenvalues of $-\tildeDeltao$ in $\tildeOmega$.

The following is the main result of this subsection.

\begin{prop}\label{prop: eigfct of Delta}
For all $t_0,\tau_0 > 0$ satisfying (\ref{eq: def tau}) and $k \in \bbN$ there exist $C > 0$ and $\varepsilon_0 > 0$ such that for all $\varepsilon \in (0,\varepsilon_0]$ we have the following estimates:
\begin{enumerate}[(i)]
	\item (Eigenvalue)
	\[
		\left|\tilde{\lambda}_k - \lambdao_k \right| \leq C \delta^{1/6}
	\]
	\item (Eigenfunction)
	\[	
		\left|\left|\tilde{u}_k - \uo_k \right|\right|_{L^2(\tildeOmega)} \leq C \delta^{1/12}
	\]
\end{enumerate}
\end{prop}

In particular, combining (\ref{eq: eigenvalue gap ODE}), \Cref{prop: Separation of variables}(i), and \Cref{prop: eigfct of Delta}(i), we obtain, for all $\varepsilon >0$ small enough,
\begin{equation}\label{eq: eigenvalue gap PDE}
	\tilde{\lambda}_{k+1}-\tilde{\lambda}_k = a_{k+1}-a_k + O\big(\delta^{1/6}\big),
\end{equation}
where $0 > -a_1 > -a_2 > \dots$ are the zeros of the Airy function $\Ai$. In particular, the first $k$ eigenvalues of $\tildeDelta$, and hence of $\Delta$, are all simple for all $\varepsilon$ small enough (depending on $k$).

We point out that in \Cref{prop: eigfct of Delta}(ii) one can replace to usual $L^2$-norm by either one of the weighted $L^2$-norms $||\cdot||_{\tilde{L}^2(\tildeOmega)}$ or $||\cdot||_{\tilde{L}_\circ^2(\tildeOmega)}$ because all these norms are uniformly equivalent (with a constant only depending on $t_0$, $\tau_0$, and the metric $g$).

\Cref{prop: eigfct of Delta} will be consequence of \Cref{lem: perturbation} since the coefficients of $\tildeDelta$ and $\tildeDeltao$ are very close. To make this precise, we need the following basic claim. In its formulation, $\nabla_y$ denotes the standard gradient in $\bbR^{n-1}$. We also remind the reader that, for all $i \in \bbN$, $\tilde{u}_i,\uo_i \in H^2(\tildeOmega) \cap H_0^1(\tildeOmega)$ due to \Cref{rem: H^2 up to boundary}.

\begin{lem}\label{lem: tilde(Delta)-tilde(Delta)_0}
For all $\tilde{u} \in H^2(\tildeOmega) \cap H_0^1(\tildeOmega)$ we have
\begin{align*}
	\left|\left \langle \left(\tildeDelta-\tildeDeltao\right)\tilde{u},\tilde{u} \right \rangle_{\tilde{L}^2(\tildeOmega)} \right| \leq &
	C \bigg[\delta^{2/3}\varepsilon^{-1}||\nabla_y u||^2_{\tilde{L}^2(\tildeOmega)} 
	+ \delta^{1/3}||\partial_x \tilde{u}||_{\tilde{L}^2(\tildeOmega)}||\nabla_y \tilde{u}||_{\tilde{L}^2(\tildeOmega)} \\
	&+
	\left(\delta^{1/3}||\partial_x u||_{\tilde{L}^2(\tildeOmega)}+\delta^{2/3}\varepsilon^{-1}||\nabla_y u||_{\tilde{L}^2(\tildeOmega)}\right) ||u||_{\tilde{L}^2(\tildeOmega)} \bigg].
\end{align*}
The same estimate holds with $\tilde{L}_\circ^2(\tildeOmega)$ instead of $\tilde{L}^2(\tildeOmega)$.
\end{lem}

The proof is just a straightforward calculation. 

\begin{proof}
Using (\ref{eq: def Laplace}) and (\ref{eq: def jac and jac0}) we write
\begin{align*}
	\Delta =& \frac{1}{\jac}\bigg(\partial_t\Big(\jac \partial_t \Big)
	+\partial_{s_i}\Big(\jac g^{ij}\partial_{s_j}\Big)
	+ \partial_{t}\Big(\jac g^{ti}\partial_{s_i}\Big)
	+ \partial_{s_i}\Big(\jac g^{ti}\partial_{t}\Big) \bigg) \\
	=& \partial^2_{tt}+g^{ij}\partial^2_{s_i s_j}+2g^{ti}\partial^2_{s_i t}
	+	\frac{\partial_t \jac}{\jac}\partial_t + \frac{\partial_{s_i}\left(\jac g^{ij}\right)}{\jac}\partial_{s_j}+
	\frac{\partial_{t}\left(\jac g^{ti}\right)}{\jac}\partial_{s_i}+\frac{\partial_{s_i}\left(\jac g^{ti}\right)}{\jac}\partial_{t}.
\end{align*}
Together with the definition (\ref{eq: def Laplace_0}) of $\Delta_\circ$, this gives
\[
	\Delta-\Delta_\circ = 
	\left(g^{ij}-g_\circ^{ij}\right)\partial^2_{s_i s_j}+2g^{ti}\partial^2_{s_i t}+O(1)\partial_t + O(1)\partial_{s_i}.
\]
Recall the change of variables $x \coloneq \delta^{-1/3}(t_0-t)$ and $y \coloneq s/\varepsilon$ from (\ref{eq: change of variables}), so that $\partial_x = -\delta^{1/3}\partial_t$ and $\partial_{y_i}=\varepsilon \partial_{s_i}$. Thus, by the definition (\ref{eq: def tilde Laplacians}) of $\tildeDelta$ and $\tildeDeltao$,
\begin{align}\label{eq: tilde(Delta)-tilde(Delta)_0 1}
	\tildeDelta-\tildeDeltao =& \delta^{2/3}\Delta-\delta^{2/3}\Delta_\circ \notag \\
	=& \delta^{2/3}\varepsilon^{-2}\left(g^{ij}-g_\circ^{ij}\right)\partial^2_{y_iy_j}
	- 2\delta^{1/3}\varepsilon^{-1}g^{ti}\partial^2_{y_i x}+O\big(\delta^{1/3}\big)\partial_x+O\big(\delta^{2/3}\varepsilon^{-1} \big)\partial_{y_j},
\end{align}
where $g_{ij}$ continues to denote $g(\partial_{s_i},\partial_{s_j})$, and \emph{not} $g(\partial_{y_i},\partial_{y_j})$. 

Now, using integration by parts, we have
\begin{align*}
	\int_{\tildeOmega}\left(g^{ij}-g_\circ^{ij}\right)\partial^2_{y_iy_j}(\tilde{u})\tilde{u} \jac \, dx dy =& - \int_{\tildeOmega}\partial_{y_j}(\tilde{u})\partial_{y_i}\Big(\left(g^{ij}-g_\circ^{ij}\right) \tilde{u}\jac \Big) \, dx dy \\
	=& -\int_{\tildeOmega}\partial_{y_j}(\tilde{u})\partial_{y_i}(\tilde{u})\left(g^{ij}-g_\circ^{ij}\right)\jac \, dx dy  \\
	&- \int_{\tildeOmega}\partial_{y_j}(\tilde{u})\tilde{u}\partial_{y_i}\Big(\left(g^{ij}-g_\circ^{ij}\right)\jac \Big) \, dx dy.
\end{align*}
However, since $g_{ij}$ and $\jac$ are smooth for $(s,t) \in  B_\varepsilon(0) \times [-\tau_0,t_0] $, we have
\[
	\left(g^{ij}-g_\circ^{ij}\right)=O(\varepsilon)
	\quad \text{and} \quad
	\partial_{y_i}\Big(\left(g^{ij}-g_\circ^{ij}\right)\jac \Big)
	=\varepsilon \partial_{s_i}\Big(\left(g^{ij}-g_\circ^{ij}\right)\jac \Big)	
	=O(\varepsilon)
\]
because, by definition, $(g_\circ)_{ij}(t) = g_{ij}(0,t)$. Thus
\begin{equation}\label{eq: tilde(Delta)-tilde(Delta)_0 2}
	\left| \left\langle \left(g^{ij}-g_\circ^{ij}\right)\partial^2_{y_iy_j}(\tilde{u}),\tilde{u} \right\rangle_{\tilde{L}^2(\tildeOmega)} \right| \leq C \varepsilon ||\nabla_y \tilde{u}||_{\tilde{L}^2(\tildeOmega)} \left( ||\nabla_y \tilde{u}||_{\tilde{L}^2(\tildeOmega)}+||\tilde{u}||_{\tilde{L}^2(\tildeOmega)} \right)
\end{equation}
Keeping in mind that $g_{ti}(0,t)=0$ for all $t \in [-\tau_0,t_0]$, with the same argument we find
\begin{equation}\label{eq: tilde(Delta)-tilde(Delta)_0 3}
	\left| \left\langle g^{ti}\partial^2_{y_i x}(\tilde{u}),\tilde{u} \right\rangle_{\tilde{L}^2(\tildeOmega)} \right| \leq C \varepsilon ||\partial_x \tilde{u}||_{\tilde{L}^2(\tildeOmega)} \left( ||\nabla_y \tilde{u}||_{\tilde{L}^2(\tildeOmega)}+||\tilde{u}||_{\tilde{L}^2(\tildeOmega)} \right)
\end{equation}
Combining (\ref{eq: tilde(Delta)-tilde(Delta)_0 1}), (\ref{eq: tilde(Delta)-tilde(Delta)_0 2}), (\ref{eq: tilde(Delta)-tilde(Delta)_0 3}) completes the proof of \Cref{lem: tilde(Delta)-tilde(Delta)_0}.
\end{proof}

\Cref{prop: eigfct of Delta} now easily follows from repeated applications of \Cref{lem: perturbation}.

\begin{proof}[Proof of \Cref{prop: eigfct of Delta}]
We split the proof into a sequence of claims.

\smallskip
\noindent
\textit{Claim 1. We have, for all $i=1,\dots,k$ and all $\varepsilon >0$ small enough,
\[
	\big|\big|\partial_x \uo_i \big|\big|_{\tilde{L}_\circ^2(\tildeOmega)}=O\big( \delta^{-1/6}\big)
	\quad \text{and} \quad
	\big|\big|\nabla_y \uo_i \big|\big|_{\tilde{L}_\circ^2(\tildeOmega)}=O(1).
\]
}

It is possbile to improve the bound for $\partial_x \uo_i$ to $||\partial_x \uo_i||_{\tilde{L}_\circ^2(\tildeOmega)}=O(1)$ (see \Cref{lem: exponential decay}). 

\smallskip
\noindent
\textit{Proof of Claim 1.}
The bound for $\partial_x \uo_i$ was already proven in (\ref{eq: L^2 bound partial_x tilde(u)_0 - weak}). Using the observation that there exists $c \in (0,1)$ such that $\langle \Go^{-1}(t)\xi,\xi \rangle \geq c|\xi|^2$ for all $\xi \in \bbR^{n-1}$ and all $t \in [-\tau_0,t_0]$, the same proof also implies the desired bound for $\nabla_y \uo_i$. 
\hfill$\blacksquare$

\smallskip
\noindent
\textit{Claim 2. We have, for all $\varepsilon > 0$ small enough,
\[
	\tilde{\lambda}_k \leq \lambdao_k + C \delta^{1/6}.
\]
In particular, $\tilde{\lambda}_i=O\big(\delta^{-1/3}\big)$ for all $i=1, \dots, k$.}

\smallskip
\noindent
\textit{Proof of Claim 2.}
Since $\jac$ is smooth, hence Lipschitz, for $(t,s) \in [-\tau_0,t_0] \times B_\varepsilon(0)$, we have $|\jac-\jaco|=O(\varepsilon)$. Hence $\langle \cdot \, , \cdot \rangle_\circ \coloneq \langle \cdot \, , \cdot \rangle_{\tilde{L}_\circ^2(\tildeOmega)}$ and $\langle \cdot \, , \cdot \rangle\coloneq \langle \cdot \, , \cdot \rangle_{\tilde{L}^2(\tildeOmega)}$ satisfy (\ref{eq: inner products almost equivalent}) for some $\varepsilon_{\rm comp}=O(\varepsilon)$. Therefore, the upper bound in (\ref{eq: eigenvalue perturbation}) implies
\[
	\tilde{\lambda}_k - \lambdao_k \leq C \varepsilon \lambdao_k + \max_{\tilde{u} \in S_k^\circ} \frac{\left|\left \langle \left(\tildeDelta-\tildeDeltao\right)\tilde{u},\tilde{u} \right \rangle_{\tilde{L}^2(\tildeOmega)} \right| }{||\tilde{u}||^2_{\tilde{L}^2(\tildeOmega)}},
\]
where $S_k^\circ \coloneq {\rm span}\big\{\uo_1, \dots, \uo_k \big\}$. It follows from the triangle inequality that Claim 1 also holds for all $\tilde{u} \in S_k^\circ$ with $||\tilde{u}||_{\tilde{L}_\circ^2(\tilde{\Omega})}=1$. Therefore, \Cref{lem: tilde(Delta)-tilde(Delta)_0} implies that, for all $\tilde{u} \in S_k^\circ$ with $||\tilde{u}||_{\tilde{L}_\circ^2(\tilde{\Omega})}=1$, we have
\begin{equation}\label{eq: Difference of Laplacians 1}
	\left|\left \langle \left(\tildeDelta-\tildeDeltao\right)\tilde{u},\tilde{u} \right \rangle_{\tilde{L}^2(\Omega)} \right|
	=O\Big(\delta^{2/3}\varepsilon^{-1}+\delta^{1/3}\delta^{-1/6}+ \delta^{1/3}\delta^{-1/6}+\delta^{2/3}\varepsilon^{-1}\Big)	= O\big( \delta^{1/6} \big),
\end{equation}
where we also used that $||\cdot||_{\tilde{L}^2(\Omega)}$ and $|\cdot||_{\tilde{L}_\circ^2(\Omega)}$ are uniformly equivalent. This completes the proof of Claim 2 since $\lambdao_k=O\big(\delta^{-1/3}\big)$ by \Cref{prop: Separation of variables}(i) and \Cref{prop: Model ODE and Airy}(i).
\hfill$\blacksquare$

We can now show that the analog of Claim 1 also holds for the eigenfunctions of $\tildeDelta$.

\smallskip
\noindent
\textit{Claim 3. We have, for all $i=1,\dots,k$ and all $\varepsilon >0$ small enough,
\[
	\big|\big|\partial_x \tilde{u}_i \big|\big|_{\tilde{L}^2(\tildeOmega)}=O\big( \delta^{-1/6}\big)
	\quad \text{and} \quad
	\big|\big|\nabla_y \tilde{u}_i \big|\big|_{\tilde{L}^2(\tildeOmega)}=O(1).
\]
}

\smallskip
\noindent
\textit{Proof of Claim 3.}
Testing $-\tildeDelta \tilde{u}_i = \tilde{\lambda}_i \tilde{u}_i$ against $\frac{\jac}{\jaco(t_0)} \tilde{u}_i$ yields, as $||\tilde{u}_i||_{\tilde{L}^2(\tildeOmega)}=1$,
\[
	\int_{\tildeOmega}\frac{\jac}{\jaco(t_0)}\left(|\partial_x \tilde{u}_i|^2 + \delta^{2/3}\varepsilon^{-2}\langle G^{-1} \nabla_y \tilde{u}_i, \nabla_y \tilde{u}_i \rangle+2\delta^{1/3}\varepsilon^{-1}g^{ti}\partial_{y_i}(\tilde{u}_i) \partial_x (\tilde{u}_i) \right) \, dx dy = \tilde{\lambda}_i,
\]
where $G=(g_{ij}) \in {\rm GL}_{n-1}(\bbR)$. Keep in mind that $g_{ti}(0,t)=0$ for all $t \in [-\tau_0,t_0]$, so that 
\begin{equation}\label{eq: g^(ti)}
	g^{ti}=O(\varepsilon).
\end{equation}
Moreover, there exists $c=c(t_0,\tau_0) \in (0,1)$ such that $\langle G^{-1}\xi,\xi \rangle \geq c|\xi|^2$ for all $\xi \in \bbR^{n-1}$. Hence we obtain, for all $\varepsilon > 0$ small enough,
\[
	\frac{1}{2}\int_{\tildeOmega}\frac{\jac}{\jaco(t_0)}\left(|\partial_x \tilde{u}_i|^2 + \delta^{2/3}\varepsilon^{-2}c|\nabla_y \tilde{u}_i|^2 \right) \, dx dy 
	\leq \tilde{\lambda}_i.
\]
Finally, we know $\tilde{\lambda}_i=O\big(\delta^{-1/3}\big)$ from Claim 2 and $\varepsilon^{2}\delta^{-2/3}=\big(\delta^{1/3}\big)$ from (\ref{eq: def delta}). This implies the desired estimates.
\hfill$\blacksquare$

\smallskip
\noindent
\textit{Claim 4. We have, for all $\varepsilon > 0$ small enough,
\[
	\big| \tilde{\lambda}_k - \lambdao_k \big| = O\big(\delta^{1/6}\big).
\]}

Note that this is \Cref{prop: eigfct of Delta}(i).

\smallskip
\noindent
\textit{Proof of Claim 4.}
Using Claim 3 instead of Claim 1, and Claim 2 instead of the upper bound $\lambdao_k = O\big(\delta^{-1/3}\big)$, with the exact same argument as in the proof of Claim 2 we obtain
\begin{equation}\label{eq: Difference of Laplacians 2}
	\left|\left \langle \left(\tildeDelta-\tildeDeltao\right)\tilde{u},\tilde{u} \right \rangle_{\tilde{L}_{\circ}^2(\tildeOmega)} \right| = O\big( \delta^{1/6} \big),
\end{equation}
for all $\tilde{u} \in {\rm span}\big\{\tilde{u}_1, \dots, \tilde{u}_k\big\}$ with $||\tilde{u}||_{\tilde{L}^2(\tildeOmega)}=1$. So, as in the proof of Claim 2, Claim 4 follows from the lower bound in (\ref{eq: eigenvalue perturbation}).
\hfill$\blacksquare$

Observe that, for all $\varepsilon > 0$ small enough, $\lambdao_{i+1}-\lambdao_i$ is uniformly bounded from below for all $i=1,\dots,k$. Indeed, this follows from \Cref{prop: Separation of variables}(i) and (\ref{eq: eigenvalue gap ODE}). We also recall that we can choose the $\varepsilon_{\rm comp}$ in the assumption (\ref{eq: inner products almost equivalent}) for \Cref{lem: perturbation} such that $\varepsilon_{\rm comp}=O(\varepsilon)$ (see the proof of Claim 2). Therefore, (\ref{eq: eigenvector perturbation}) in \Cref{lem: perturbation} and (\ref{eq: Difference of Laplacians 1}), (\ref{eq: Difference of Laplacians 2}) imply
\[
	||\tilde{u}_k - \uo_k||^2_{\tilde{L}_\circ^2(\tildeOmega)}=O\big(\delta^{1/6}\big).
\]
Since $||\cdot||_{L^2(\tildeOmega)}$ and $||\cdot||_{\tilde{L}_\circ^2(\tildeOmega)}$ are uniformly equivalent, this is exactly \Cref{prop: eigfct of Delta}(ii), and thus completes the proof.
\end{proof}

For the proof of \Cref{Main Theorem} we will also need the following exponential decay estimate.

\begin{lem}[Exponential decay]\label{lem: exponential decay}
For all $t_0,\tau_0 > 0$ satisfying (\ref{eq: def tau}) and $k \in \bbN$ there exist $C > 0$, $x^\ast > 0$, and $\varepsilon_0 > 0$ such that for all $\varepsilon \in (0,\varepsilon_0]$ we have
\begin{equation}\label{eq: exponential decay}
	\int_{B_1(0)} \tilde{u}_k^2(x,y)\, dy \leq Ce^{-x}
	\quad \text{for all } x \in \big[x^\ast,\delta^{-1/3}(t_0+\tau_0)\big].
\end{equation}
In particular, for every polynomial $Q \in \bbR[X]$ there exists a constant $C_Q < \infty$ such that
\begin{equation}\label{eq: weighted integral}
	\int_{\tildeOmega}|Q|(x) \tilde{u}_k^2(x,y)  \, dx dy \leq C_Q
\end{equation}
and
\begin{equation}\label{eq: weighted difference integral}
	\int_{\tildeOmega}|Q|(x) \left| \tilde{u}_k^2-(\uGuess_k)^2\right|(x,y)  \, dx dy \leq C_Q\delta^{1/12} ,
\end{equation}
where $\uGuess_k$ is the function defined in (\ref{eq: u Guess}).
\end{lem}

\begin{proof}
Observe that (\ref{eq: weighted integral}) follows easily from (\ref{eq: exponential decay}), and that the analogue of (\ref{eq: weighted integral}) also holds for $\uGuess_k$ due to the definition (\ref{eq: u Guess}) of $\uGuess_k$ and Claim 3 in the proof of \Cref{prop: Model ODE and Airy}. Moreover, for any polynomial $Q \in \bbR[X]$, we have
\[
    \int_{\tildeOmega}|Q|(x) \left| \tilde{u}_k^2-(\uGuess_k)^2\right|(x,y)  \, dx dy
    \leq \big(\left|\left|Q\tilde{u}_k \right|\right|_{L^2}
    +\left|\left|Q\uGuess_k \right|\right|_{L^2}\big)
    \left|\left|\tilde{u}_k-\uGuess_k \right|\right|_{L^2}
\]
because of the third binomial formula, the Cauchy-Schwarz inequality, and the triangle inequality. Because we already know $||\tilde{u}_k-\uGuess_k ||_{L^2} \leq C \delta^{1/12}$ thanks to \Cref{prop: Separation of variables}(ii) and \Cref{prop: eigfct of Delta}(ii), the weighted $L^2$-bound (\ref{eq: weighted difference integral}) thus follows from (\ref{eq: weighted integral}), which in turn easily follows from (\ref{eq: exponential decay}). Therefore, it suffices to prove (\ref{eq: exponential decay}).

Consider the slice-energy function $\mathfrak{e}_k$ defined, for $x \in [0,\delta^{-1/3}(t_0+\tau_0)]$, by 
\[
	\mathfrak{e}_k(x) \coloneq \frac{1}{2}||\tilde{u}_k||_{L^2(\{x\} \times B_1)}^2 \coloneq \frac{1}{2}||\tilde{u}_k(x,\cdot)||_{L^2(B_1)} ^2
	\coloneq \frac{1}{2}\int_{B_1(0)} \tilde{u}_k^2(x,y) \, dy.
\]
Below, we will show by a straightforward albeit tedious calculation
\begin{equation}\label{eq: mass decay 0}
	\mathfrak{e}_k^{\prime \prime}(x) \geq \Big(\big(1+O(\varepsilon)\big)\delta^{2/3}\varepsilon^{-2}\mu_1\big(t_0-\delta^{1/3}x\big)-\tilde{\lambda}_k \Big)2\mathfrak{e}_k(x).
\end{equation}
We first explain how (\ref{eq: exponential decay}) follows from (\ref{eq: mass decay 0}). 

Namely, from \Cref{prop: Model ODE and Airy}(i), \Cref{prop: Separation of variables}(i) and \Cref{prop: eigfct of Delta}(i) we know 
\[
	\tilde{\lambda}_k = \delta^{2/3}\varepsilon^{-2}\mu_1(t_0)+a_k + O(\delta^{1/6}).
\]
Moreover, by Taylor and the definition (\ref{eq: def delta}) of $\delta$ we know (see the proof of \Cref{lem: Model ODE is perturbed Airy} for details)
\[
	\delta^{2/3}\varepsilon^{-2}\mu_1(t_0-\delta^{1/3}x)=\delta^{2/3}\varepsilon^{-2}\mu_1(t_0)+x+O\big(\delta^{1/3}x^2\big).
\]
Finally, $x \mapsto \mu_1(t_0-\delta^{1/3}x)$ is monotonically increasing for $x \in [0,\delta^{-1/3}t_0]$ due to \Cref{lem: Monotonicity of mu_1}, and $\mu_1(t) \geq \mu_1(t_0/4)$ for $t \in [-\tau_0,0]$ due to the choice (\ref{eq: def tau}) of $\tau_0$. Therefore, if we set $x_0 \coloneq a_k+10$, then for all $x \in [x_0,\delta^{-1/3}(t_0+\tau_0)]$ and $\varepsilon>0$ small enough, we have
\[
	\big(1+O(\varepsilon)\big)\delta^{2/3}\varepsilon^{-2}\mu_1(t_0-\delta^{1/3}x)-\tilde{\lambda}_k 
	\geq  \big(1+O(\varepsilon)\big)\delta^{2/3}\varepsilon^{-2}\mu_1(t_0-\delta^{1/3}x_0)-\tilde{\lambda}_k
	\geq  2,
\]
and consequently $\mathfrak{e}_k^{\prime \prime}(x) \geq 4\mathfrak{e}_k(x)$ for all $x \in [x_0,\delta^{-1/3}(t_0+\tau_0)]$ by (\ref{eq: mass decay 0}). Appealing to the Sturm comparison theorem we obtain, for all $x_0 \leq x^\ast \leq x \leq \delta^{-1/3}(t_0+\tau_0)$,
\[
	\mathfrak{e}_k(x^\ast) \geq \cosh(2(x-x^\ast))\mathfrak{e}_k(x) \geq \frac{1}{2}e^{2(x-x^\ast)}\mathfrak{e}_k(x)
\]
because $\mathfrak{e}_k\big(\delta^{-1/3}(t_0+\tau_0)\big)=0$ due to the Dirichlet boundary condition. Finally, there exists $x^\ast \in [x_0,x_0+1]$ such that 
\[
	\mathfrak{e}_k(x^\ast) \leq \int_{x_0}^{x_0+1} \mathfrak{e}_k(x) \, dx \leq \frac{1}{2} ||\tilde{u}_k||_{L^2(\tildeOmega)}^2 \leq C.
\]
Together the two previous estimates imply $\mathfrak{e}(x) \leq C e^{-2x}$ for all $x \in [x^\ast,\delta^{-1/3}(t_0+\tau_0)]$, which is exactly the desired decay estimate (\ref{eq: exponential decay}).

It remains to prove (\ref{eq: mass decay 0}). In the remainder, all integrals will be understood to be over $\{x\} \times B_1(0)$, but we will only write $B_1(0)$ for ease of notation.

Clearly, for all $x \in [0,\delta^{-1/3}(t_0+\tau_0)]$, we have
\begin{align}\label{eq: mass decay - second derivative}
	\mathfrak{e}_k^{\prime \prime}(x)
	=&\int_{B_1(0)}\Big(\partial^2_{xx}(\tilde{u}_k)\tilde{u}+|\partial_x \tilde{u}_k|^2 \Big) \, dy \notag \\
	=&-2\tilde{\lambda}_k \mathfrak{e}_k(x)+\int_{B_1(0)}\Big(\big(\partial^2_{xx}(\tilde{u}_k)-\tildeDelta\tilde{u}_k\big)\tilde{u}_k+|\partial_x \tilde{u}_k|^2 \Big) \, dy,
\end{align}
where in the second equality we used $\tildeDelta \tilde{u}_k=-\tilde{\lambda}_k \tilde{u}_k$. We recall
\begin{equation}\label{eq: mass decay - Laplacian}
\begin{split}
	\tildeDelta = \frac{1}{\jac}\bigg(&\jac \partial^2_{xx}+\partial_x(\jac)\partial_x
	+\delta^{2/3}\varepsilon^{-2}\partial_{y_i}\Big(\jac g^{ij}\partial_{y_j}\Big) \\
	&+\delta^{1/3}\varepsilon^{-1}\partial_{x}\Big(\jac g^{ti}\partial_{y_i}\Big)
	+\delta^{1/3}\varepsilon^{-1}\partial_{y_i}\Big(\jac g^{ti}\partial_{x}\Big) \bigg),
\end{split}
\end{equation}
where $g_{ij}$ denotes $g(\partial_{s_i},\partial_{s_j})$, and \emph{not} $g(\partial_{y_i},\partial_{y_j})$.

We will now estimate a number of integrals. Throughout, we will use that
\begin{equation}\label{eq: mass decay - help}
	\partial_x({\rm fct})=O\big(\delta^{1/3}\big)
	\quad \text{and} \quad
	\partial_{y_i}({\rm fct})=O(\varepsilon)
\end{equation}
for ${\rm fct}$ being either of the functions $\jac$, $g^{ij}$, $g^{it}$, or a product thereof. Indeed, this simply follows from the fact that they are all smooth for $(s,t) \in B_\varepsilon(0) \times [-\tau_0,t_0] $, and $\partial_x=-\delta^{1/3}\partial_t$ and $\partial_{y_i}=\varepsilon \partial_{s_i}$ due to the definition (\ref{eq: change of variables}) of the change of variables. 

First, (\ref{eq: mass decay - help}) together with the Poincaré inequality on $B_1(0)$ immediately yields
\begin{align}\label{eq: mass decay 1}
	\left|-\int_{B_1(0)}\frac{\partial_x(\jac)}{\jac}\partial_x(\tilde{u}_k)\tilde{u}_k \, dy \right| 
	\leq & C\delta^{1/3}||\nabla_y \tilde{u}_k||_{L^2(B_1)}||\partial_x \tilde{u}_k||_{L^2(B_1)} \notag \\
	\leq & C \delta^{2/3}||\nabla_y \tilde{u}_k||_{L^2(B_1)}^2 +\frac{1}{4}||\partial_x \tilde{u}_k||_{L^2(B_1)}^2,
\end{align}
where for the second inequality we also used that $2ab \leq a^2+b^2$ for all $a,b \in \bbR$.

Second, from the product rule and integration by parts we get
\begin{align*}
	-\int_{B_1(0)}\frac{1}{\jac}\partial_{x}\Big(\jac g^{ti}\partial_{y_i}(\tilde{u}_k)\Big)\tilde{u}_k \, dy 
	=& -\int_{B_1(0)}\frac{\partial_{x}\big(\jac g^{ti} \big)}{\jac}\partial_{y_i}(\tilde{u}_k)\tilde{u}_k \, dy  \\
	&+ \int_{B_1(0)}\partial_{x}(\tilde{u}_k)\partial_{y_i}\left(g^{ti}\tilde{u}_k\right) \, dy .
\end{align*}
Together with (\ref{eq: mass decay - help}), $g^{it}=O(\varepsilon)$ (see (\ref{eq: g^(ti)})), and the Poincaré inequality, we deduce
\begin{align}\label{eq: mass decay 2}
	\left| -\int_{B_1(0)}\frac{1}{\jac}\partial_{x}\Big(\jac g^{ti}\partial_{y_i}(\tilde{u}_k)\Big)\tilde{u}_k \, dy  \right|
	\leq &  C \delta^{1/3}||\nabla_y \tilde{u}_k||_{L^2(B_1)}^2 \notag\\
	&+C\varepsilon||\nabla_y \tilde{u}_k||_{L^2(B_1)}||\partial_x \tilde{u}_k||_{L^2(B_1)} \notag \\
	\leq & C\delta^{1/3}||\nabla_y \tilde{u}_k||_{L^2(B_1)}^2	+ \frac{\delta^{-1/3}\varepsilon^2}{4}||\partial_x \tilde{u}_k||_{L^2(B_1)}^2.
\end{align}
Similarly, using the divergence theorem, $g^{it}=O(\varepsilon)$, and the Poincaré inequality, we can bound
\begin{align}\label{eq: mass decay 3}
	\left| -\int_{B_1(0)}\frac{1}{\jac}\partial_{y_i}\Big(\jac g^{ti}\partial_{x}(\tilde{u}_k)\Big)\tilde{u}_k \, dy \right|
	\leq & C\varepsilon||\nabla_y \tilde{u}_k||_{L^2(B_1)}||\partial_x\tilde{u}_k||_{L^2(B_1)} \notag \\
	\leq &  C\delta^{1/3}\varepsilon ||\nabla_y \tilde{u}_k||_{L^2(B_1)}^2+\frac{\delta^{-1/3}\varepsilon}{4}||\partial_x\tilde{u}_k||_{L^2(B_1)} ^2 .
\end{align}
Finally, we have, again by the divergence theorem,
\begin{align}
	-\int_{B_1(0)}\frac{1}{\jac}\partial_{y_i}\Big(\jac g^{ij}\partial_{y_j}(\tilde{u}_k)\Big) \tilde{u}_k \, dy
	=& \int_{B_1(0)}\Big(\jac g^{ij}\partial_{y_j}(\tilde{u}_k)\Big) \partial_{y_i}\left(\frac{1}{\jac}\tilde{u}_k\right) \, dy \notag \\
	=& \int_{B_1(0)} g^{ij}\partial_{y_j}(\tilde{u}_k) \partial_{y_i}(\tilde{u}_k)\, dy \label{eq: mass decay 4} \\
	&-\int_{B_1(0)}\frac{\partial_{y_i}(\jac)}{\jac} g^{ij}\partial_{y_j}(\tilde{u}_k)\tilde{u}_k \, dy. \label{eq: mass decay 5}
\end{align}
Since $g_{ij}=(g_\circ)_{ij}+O(\varepsilon)$, we can also bound bound the first summand (\ref{eq: mass decay 4}) by
\begin{equation}\label{eq: mass decay 6}
	\int_{B_1(0)} g^{ij}\partial_{y_j}(\tilde{u}_k) \partial_{y_i}(\tilde{u}_k)\, dy
	\geq \int_{B_1(0)} g_\circ^{ij}\partial_{y_j}(\tilde{u}_k) \partial_{y_i}(\tilde{u}_k)\, dy -C\varepsilon ||\nabla_y \tilde{u}_k||_{L^2(B_1)}^2.
\end{equation}
Due to (\ref{eq: mass decay - help}) and the Poincaré inequality, the second summand (\ref{eq: mass decay 5}) is bounded by
\begin{equation}\label{eq: mass decay 7}
	\left| -\int_{B_1(0)}\frac{\partial_{y_i}(\jac)}{\jac}\jac g^{ij}\partial_{y_j}(\tilde{u}_k)\tilde{u}_k \, dy \right| \leq C\varepsilon ||\nabla_y \tilde{u}_k||_{L^2(B_1)}^2.
\end{equation}
Combining (\ref{eq: mass decay - Laplacian}), and (\ref{eq: mass decay 1})-(\ref{eq: mass decay 7}) implies
\begin{align*}
	\int_{B_1(0)} \big(\partial^2_{xx}(\tilde{u}_k)-\tildeDelta \tilde{u}_k\big)\tilde{u}_k \, dy \geq& \, \delta^{2/3}\varepsilon^{-2} \int_{B_1(0)} g_\circ^{ij}\partial_{y_j}(\tilde{u}_k) \partial_{y_i}(\tilde{u}_k)\, dy \\
	& - C \big(\delta^{2/3}+\delta^{2/3}\varepsilon^{-1}+\delta^{2/3}+ \delta^{2/3}\varepsilon^{-1}\big)||\nabla_y \tilde{u}_k||_{L^2(B_1)}^2 \\
	&-\left(\frac{1}{4}+\frac{\varepsilon}{4}+\frac{1}{4}\right)||\partial_x \tilde{u}_k||_{L^2(B_1)}^2.
\end{align*}
As $\Go^{-1}=(g_\circ^{ij})$ is uniformly elliptic, we can simplify this to
\begin{equation}\label{eq: mass decay 8}
\begin{split}
	\int_{B_1(0)} \big(\partial^2_{xx}(\tilde{u}_k)-\tildeDelta \tilde{u}_k\big)\tilde{u}_k \, dy \geq & (1-C\varepsilon) \delta^{2/3}\varepsilon^{-2} \int_{B_1(0)} \langle \Go^{-1}\nabla_y \tilde{u}_k, \nabla_y \tilde{u}_k \rangle\, dy \\
	&-\frac{3}{4}||\partial_x \tilde{u}_k||_{L^2(B_1)}^2.
\end{split}
\end{equation}
However, by definition, $\mu_1(t)$ is the first eigenvalue of $-\Div_y\big(\Go^{-1}(t)\nabla_y \big)$. Therefore, keeping in mind that $t=t_0-\delta^{1/3}x$ by (\ref{eq: change of variables}), the desired estimate (\ref{eq: mass decay 0}) follows from (\ref{eq: mass decay - second derivative}) and (\ref{eq: mass decay 8}). This completes the proof.
\end{proof}

\subsection{Proof of \Cref{Main Theorem}}\label{subsec: Proof of Main Thm}

We are now finally in the position to prove \Cref{Main Theorem}. As in the formulation of \Cref{Main Theorem}, we let $X$ be a negatively curved Hadamard manifold, and denote by $\Omega_\varepsilon \coloneq \Omega(\varepsilon;-\tau_0,t_0) \subseteq X$ the domain defined in \Cref{def: Domain}. Finally, we denote by $(u_i)_{i \in \bbN} \subseteq L^2(\Omega_\varepsilon)$ a complete orthonormal basis of eigenfunctions of $-\Delta$ in $\Omega_\varepsilon \subseteq X$ with Dirichlet boundary data (labelled in such a way that the associated sequence of eigenvalues $(\lambda_i)_{i \in \bbN}$ is non-decreasing). 

We start with the following asymptotic expansion for the integral in (\ref{eq: core of strategy}).

\begin{cor}\label{cor: asymptotic integral}
For all $t_0, \tau_0 > 0$ satisfying (\ref{eq: def tau}) and $P \in C^2(\bbR)$ with $P^\prime(t_0) > 0$, there exist $C > 0$ and $\varepsilon_0 > 0$ such that for all $\varepsilon \in (0,\varepsilon_0]$ the integral
\[
	{\rm I}_\varepsilon \coloneq \int_{\Omega_\varepsilon}P(t) \left(u_2^2(s,t)-u_1^2(s,t) \right) \, d{\rm vol}_g(s,t)
\]
satisfies
\[
	\left|\frac{\delta^{-1/3}{\rm I}_\varepsilon}{P^\prime(t_0)}+\frac{2}{3}(a_2-a_1)\right|
	\leq C \delta^{1/12},
\]
where $0 > -a_1 > -a_2 > \dots$ are the zeros of the Airy function $\Ai$. In particular, ${\rm I}_\varepsilon < 0$ for all $\varepsilon > 0$ small enough.
\end{cor}

In particular, this holds for the potential $P$ defined in \Cref{def: Potential}.

\begin{proof}
Observe that, for all $i \in \bbN$, $\tilde{u}_i(x,y) \coloneq n_i u_i(\varepsilon y,t_0-\delta^{1/3}x)$ is an eigenfunction of $-\tildeDelta$ in $\tildeOmega_\varepsilon$ with Dirichlet boundary data. Here $n_i \in \bbR_{>0}$ is a renormalization constant. From the definition (\ref{eq: change of variables}) of the change of variables we have
\[
	dt=-\delta^{1/3}dx
	\quad \text{and} \quad 
	ds=\varepsilon^{n-1}dy.
\] 
It then easily follows from the definition (\ref{eq: def weighted L^2 norm xy}) of $\tilde{L}^2(\tildeOmega_\varepsilon)$ that $(\tilde{u}_i)_{i \in \bbN} \subseteq \tilde{L}^2(\tildeOmega_\varepsilon)$ form a complete orthonormal basis of eigenfunctions if we choose the renormalization constant by requiring, for all $i \in \bbN$,
\[
	n_i^2 = \jaco(t_0)\delta^{1/3}\varepsilon^{n-1}.
\]
Using that the $u_i \in L^2(\Omega_\varepsilon)$ are $L^2$-normalized, we can thus rewrite the integral ${\rm I}_\varepsilon$ as
\begin{align*}
	{\rm I}_\varepsilon
	= & \int_{\Omega_\varepsilon}P(t)\left(u_2^2(s,t)-u_1^2(s,t) \right) \jac(t,s) \, dsdt \\
	= & \int_{\Omega_\varepsilon}\big(P(t)-P(t_0)\big)\left(u_2^2(s,t)-u_1^2(s,t) \right) \jac(t,s) \, dsdt \\
	= & \delta^{1/3}\varepsilon^{n-1} \int_{\tildeOmega_\varepsilon}\big(P(t_0-\delta^{1/3}x)-P(t_0)\big)\left(\frac{\tilde{u}_2^2(x,y)}{n_2^2}-\frac{\tilde{u}_1^2(x,y)}{n_1^2}\right) \jac(\varepsilon y,t_0-\delta^{1/3}x) \, dxdy \\
	= & \int_{\tildeOmega_\varepsilon}\big(P(t_0-\delta^{1/3}x)-P(t_0)\big)\left(\tilde{u}_2^2(x,y)-\tilde{u}_1^2(x,y)\right) \frac{\jac(\varepsilon y,t_0-\delta^{1/3}x)}{\jaco(t_0)} \, dxdy.
\end{align*}
Now, by Taylor,
\[
	P(t_0-\delta^{1/3}x)-P(t_0) = -P^\prime(t_0)\delta^{1/3}x+O\big(\delta^{2/3}x^2 \big).
\]
Similarly, as $\jac$ is smooth, hence Lipschitz, for $(s,t) \in  B_\varepsilon(0) \times [-\tau_0,t_0]$, we have
\[
	\frac{\jac(\varepsilon y,t_0-\delta^{1/3}x)}{\jaco(t_0)}=1+O\big(\varepsilon +  \delta^{1/3}x\big).
\]
Thus
\begin{align*}
	\frac{\delta^{-1/3}{\rm I}_\varepsilon}{P^\prime(t_0)} 
	=& \int_{\tildeOmega_\varepsilon}
	\left(-x+O\big(\delta^{1/3}x^2\big)\right)
	\left(\tilde{u}_2^2(x,y)-\tilde{u}_1^2(x,y)\right) 
	\left( 1+O\big(\varepsilon +  \delta^{1/3}x\big)\right) \, dxdy \\
	=& \int_{\tildeOmega_\varepsilon}
	(-x)
	\left(\tilde{u}_2^2(x,y)-\tilde{u}_1^2(x,y)\right)\, dxdy + O\big(\delta^{1/3}\big),
\end{align*}
where in the second equality we used (\ref{eq: weighted integral}) from \Cref{lem: exponential decay}, and the fact that $\varepsilon=O\big(\delta^{1/2}\big)$ due to the definition (\ref{eq: def delta}) of $\delta$. From (\ref{eq: weighted difference integral}) in \Cref{lem: exponential decay}, the definition (\ref{eq: u Guess}) of $\uGuess_i$, and \Cref{prop: Model ODE and Airy}(ii), we thus finally obtain
\begin{align*}
	\int_{\tildeOmega_\varepsilon}
	(-x)
	\left(\tilde{u}_2^2-\tilde{u}_1^2\right)(x,y)\, dxdy
	= & \int_{\tildeOmega_\varepsilon}
	(-x)
	\left((\uGuess_2)^2-(\uGuess_1)^2\right)(x,y)\, dxdy + O\big(\delta^{1/12} \big) \\
	=& \int_0^{(t_0+\tau_0)\delta^{-1/3}}(-x)\left(\tilde{h}_2^2(x)-\tilde{h}_1^2(x)\right) \, dx + O\big(\delta^{1/12} \big)  \\
	=& -\frac{2}{3}(a_2-a_1) + O\big(\delta^{1/12} \big) ,
\end{align*}
where $\tilde{h}_1$ and $\tilde{h}_2$ are the first two eigenfunctions of the model ODE (\ref{eq: Model ODE}). This completes the proof.
\end{proof}

\Cref{Main Theorem} is now immediate.

\begin{proof}[Proof of \Cref{Main Theorem}]
Fix a diameter $D_0 > 0$. Set $t_{\pm } \coloneq \frac{10 \pm 1}{10}D_0$, and fix some
\[
	\tau_0 \in \left(0,\frac{1}{10}D_0\right)
	\quad \text{such that} \quad
	\mu_1(-\tau_0) \geq \mu_1(t_{-}/4),
\]
so that, due to \Cref{lem: Monotonicity of mu_1}, the condition (\ref{eq: def tau}) is satisfied for all $t_0 \in [t_{-},t_{+}]$ (recall that such $\tau_0$ exists thanks to \Cref{lem: Monotonicity of mu_1}). Then, there exists $\varepsilon_0 > 0$ (only depending on $D_0$) such that the domain $\Omega(\varepsilon;-\tau_0,t_0)$  defined in \Cref{lem: convex domain} is convex for all $\varepsilon \in (0,\varepsilon_0]$ and all $t_0 \in [t_{-},t_{+}]$ (see \Cref{rem: uniform constants}). We also fix $t_{P} \coloneq \frac{D_0}{3}$, and let $P$ be the potential defined in \Cref{def: Potential} for this $t_P$. Then, after potentially decreasing $\varepsilon_0$, $P$ is convex on $\Omega(\varepsilon;-\tau_0,t_0)$ for all $\varepsilon \in (0,\varepsilon_0]$ and all $t_0 \in [t_{-},t_{+}]$ because of \Cref{lem: convex potential}. Finally, it follows from \Cref{cor: asymptotic integral} that, again after potentially decreasing $\varepsilon_0$, the integral in (\ref{eq: core of strategy}) is negative for all $\varepsilon \in (0,\varepsilon_0]$ and all $t_0 \in [t_{-},t_{+}]$. Therefore, for all $\varepsilon \in (0,\varepsilon_0]$ and all $t_0 \in [t_{-},t_{+}]$, there exists a convex potential on $\Omega(\varepsilon;-\tau_0,t_0)$ whose fundamental gap is strictly smaller than that of the constant potential as a consequence of the Hellmann--Feynman identity (see the discussion after (\ref{eq: core of strategy})).

It remains to show that one of the domains $\Omega(\varepsilon;-\tau_0,t_0)$ has diameter exactly $D_0$. To see this observe that
\[
	t_0+\tau_0 \leq {\rm diam}\big(\Omega(\varepsilon;-\tau_0,t_0)\big) \text{ for all }\varepsilon > 0
	\quad \text{and} \quad
	\lim_{\varepsilon \to 0}{\rm diam}\big(\Omega(\varepsilon;-\tau_0,t_0)\big)=t_0+\tau_0.
\]
So, due to the definitions of $t_{\pm}$ and $\tau_0$, we have, after potentially decreasing $\varepsilon_0$,
\[
	{\rm diam}\big(\Omega(\varepsilon_0;-\tau_0,t_{-})\big) < D_0
	\quad \text{and} \quad
	{\rm diam}\big(\Omega(\varepsilon_0;-\tau_0,t_{+})\big) > D_0.
\]
Since the map $t \mapsto {\rm diam}\big(\Omega(\varepsilon;-\tau_0,t)\big)$ is clearly continuous, we find $t_0 \in (t_{-},t_{+})$ such that ${\rm diam}\big(\Omega(\varepsilon;-\tau_0,t_0)\big)=D_0$. This completes the proof.
\end{proof}



\end{document}